\documentclass{amsart}
\usepackage{amsfonts,amsmath,amssymb,amsthm,amscd,enumerate,hyperref,comment}
\usepackage[dvipsnames]{xcolor}

\newtheorem{theorem}{Theorem}[section]
\newtheorem{lemma}[theorem]{Lemma}
\newtheorem{proposition}[theorem]{Proposition}

\theoremstyle{definition}

\theoremstyle{remark}
\newtheorem{remark}[theorem]{Remark}

\numberwithin{equation}{section}

\begin{document}
	\title
	{Quasi-Einstein manifolds with Harmonic Weyl curvature}
	
	\author{Huai-Dong Cao, Fengjiang Li$^{\dag}$ and James Siene}
	\address{Department of Mathematics\\ Lehigh University\\
		Bethlehem, PA 18015, USA.} 
	\email{huc2@lehigh.edu}
	
	\address{Mathematical Science Research Center\\ Chongqing  University of Technology\\ Chongqing, 400054, P.R. China.}
	\email{fengjiangli@cqut.edu.cn}
	
	\address{Department of Mathematics\\ Lehigh University\\
		Bethlehem, PA 18015, USA.} 
	\email{jts614@lehigh.edu}
	
 \thanks{$^{\dag}$Research partially supported by the China Scholarship Council (No.201906140158),  National Natural Science Foundation of China (No.12301062) and Natural Science Foundation of Chongqing (No.CSTB2024NSCQ-MSX0537)}

\begin{abstract}
	In this paper, we classify $n$-dimensional ($n\geq 5$) quasi-Einstein manifolds with harmonic Weyl curvature, thus extending the work of Shin \cite{Shin} in dimension four for quasi-Einstein manifolds and refining the work of He-Petersen-Wylie \cite{HPW}. As a consequence, we provide new examples of quasi-Einstein manifolds which are neither locally conformally flat nor D-flat in the sense of \cite{CC12}.
\end{abstract}

\maketitle
\date{}

\section{Introduction}

A \textit{quasi-Einstein}  manifold, as defined in Case-Shu-Wei \cite{CSW},  consists of a Riemannian manifold $(M^n,g)$ and a smooth function $f$ satisfying the equation
\begin{equation}\label{qe}
 Ric +\nabla^2f - \frac{1}{m} df \otimes df  = \lambda g.
\end{equation}
where, $m\neq 0$ and $\lambda$ are constants, and $\nabla^2f$ denotes the Hessian of $f$.  Clearly, when $f$ is constant, a quasi-Einstein manifold reduces to an Einstein manifold. In this case, it is called a trivial quasi-Einstein manifold. 

The symmetric 2-tensor  
       \[ Ric^m_f := Ric + \nabla^2f - \frac{1}{m} df \otimes df\] 
was introduced in \cite{Qian} and is called the \emph{m-Bakry-\'Emery tensor}, which extends the usual Bakry-\'Emery tensor $Ric_f := Ric + \nabla^2f$ \cite{BE}.
This tensor serves as a natural analogue of the Ricci tensor on smooth metric measure spaces and appears in the Bochner inequality for the drift Laplacian $\Delta_f:=\Delta -\nabla f \cdot \nabla $;  see  Qian \cite{Qian} and Wei-Wylie \cite{WW} for further details.

Alternatively, by setting 
$w:=e^{-f/m}$ as in \cite{CSW}, the quasi-Einstein equation \eqref{qe} can be written as
\begin{equation*}\label{qe-w}
\nabla^2\ w = \frac{w}{m} ( {Ric} - \lambda g), \qquad m\neq 0.
\end{equation*}
Such an equation appeared, for instance, in Besse \cite{Besse}. In fact, when $m>1$ is a positive integer, it follows from the work of Kim-Kim \cite{KK} that quasi-Einstein manifolds correspond to warped product Einstein manifolds. More precisely, they are exactly those $n$-dimensional manifolds which are the base of an $n+m$ dimensional Einstein warped product, i.e.,  there exists an Einstein manifold $(F^m, g_F)$, with Einstein constant $\mu$ for certain $\mu$, such that $(M \times F^m, g + e^{-2f/m}g_F)$ is an Einstein manifold with Einstein constant $\lambda$; for details, see \cite[Corollary 9.107]{Besse}, \cite{KK} and \cite{CSW}. 
For this reason,  quasi-Einstein manifolds are also referred to as \textit{$\left( \lambda ,n+m\right) $-Einstein manifolds}.  We note several special cases:
\begin{itemize}

\item[{$\bullet$}] When $m=1$, $(\lambda, n+1)$-Einstein metrics are commonly called \emph{static metrics}; see, e.g.,  \cite{Anderson, AndersonKhuri, Besse, Corvino, Israel}.

\smallskip
\item[{$\bullet$}] When $m = 2$, a quasi-Einstein manifold corresponds to a {\it static vacuum near horizon geometry} in the physics literature; see, e.g., \cite{BGKW} and \cite{Wylie 2023}.

\smallskip
\item[{$\bullet$}]  In the extended case of $m = \infty$, a quasi-Einstein manifold is simply a gradient Ricci soliton; see, e.g., \cite{Cao2010}.

\smallskip
\item[{$\bullet$}] 
When $m=2-n$, any $(\lambda, 2)$-Einstein manifold $(M^n,g,f)$ is known to be conformal to an Einstein manifold \cite{Brinkmann, JW}. Moreover, it is globally conformally equivalent to a space form \cite {CMMR}. 

\end {itemize}

Examples of non-trivial quasi-Einstein manifolds with $\lambda  < 0$ and  $\lambda = 0$ can be found in  \cite[Chapter 9]{Besse}. Non-trivial examples with $\lambda > 0$ and $m > 1$ were constructed in Lu-Page-Pope 
\cite[Section 5]{LPP}. Furthermore, $\lambda > 0$ necessarily implies that $M^n$ is compact; see \cite[Theorem 5]{Qian}.

In recent years, motivated in part by the progress in gradient Ricci solitons, there has been growing interest in studying quasi-Einstein manifolds under various conditions on the Weyl tensor $W$. For example, for $m\neq 2-n$, Catino-Mantegazza-Mazzieri-Rimoldi \cite{CMMR} have proven that a complete \emph{locally conformally flat} quasi-Einstein manifold of dimension $n\geq 3$ is locally a warped product with $(n-1)$-dimensional fibers of constant sectional curvature around any regular point of $f$.  Catino \cite{Ca} showed that a {\it generalized} quasi-Einstein manifold, for which both $m$ and $\lambda$ could be functions on $M^n$, with \emph{harmonic Weyl tensor} and $W(\nabla f,\cdot
,\cdot,\cdot)=0$ is locally a warped product with $(n-1)$-dimensional Einstein fibers around any regular point of $f$. Around the same time, He, Petersen and Wylie \cite{HPW} obtained 
both local and global classifications of warped product Einstein manifolds with\emph{ harmonic Weyl tensor} and $W(\nabla w, \cdot, \nabla w, \cdot) = 0$, or equivalently, $W(\nabla f,\cdot,\nabla f,\cdot)=0$. In addition,  
Chen-He \cite{CH} studied compact $n$-dimensional ($n\geq 4$) compact Bach-flat quasi-Einstein manifolds;  see also Huang-Wei \cite{HW} for an extension.  

It is worth noting that,  for quasi-Einstein manifolds,  having a \emph{harmonic Weyl tensor} (i.e., vanishing Cotton tensor) and $W(\nabla f,\cdot,\cdot,\cdot)=0$ implies the vanishing of the $D$-tensor first introduced in \cite{CC13} (see also \cite{CC12}),
$$\frac {m+n-2} {m}D_{ijk}=C_{ijk}-W_{ijkl}\nabla_l f,$$
where $C_{ijk}$ and $W_{ijkl}$ denote the Cotton tensor and  the Weyl tensor, respectively. When $m=\infty$, this tensor is precisely the $D$-tensor for gradient Ricci solitons.

In 2017, Shin \cite{Shin} further investigated a special class of generalized quasi-Einstein manifolds, called $(m,\rho)$-quasi-Einstein manifolds as in \cite{HW},  which satisfy the equation
$Ric^m_f=(\rho R+\lambda)g, \ \rho, \lambda\in {\mathbb R},$ where $R$ denotes the scalar curvature of $(M^n, g)$ and $m\in {\mathbb R}\setminus\{0, \pm\infty\}$.  
Building on the work of Kim \cite{Kim} on $4$-dimensional gradient Ricci solitons, Shin \cite{Shin} classified $4$-dimensional $(m,\rho)$-quasi-Einstein manifolds with harmonic Weyl curvature. However, it has remained open whether Shin's result can be extended to all dimensions $n\geq 5$ for $(m,\rho)$-quasi-Einstein manifolds, or quasi-Einstein manifolds.

In this paper, based in part on very recent progress on the (local) classification of gradient Ricci solitons with harmonic Weyl curvature \cite{Li}, \cite{Kim2} and \cite{BHS2025}, we study $n$-dimensional $(n\geq5)$ quasi-Einstein manifolds with harmonic Weyl curvature and provide both local and global characterizations of such manifolds; see Theorem \ref{local} and Theorem \ref{complete'}, respectively. In particular, our results extend Shin's work \cite{Shin} to dimensions $n\geq 5$ and refines the earlier results of Catino \cite{Ca} and He-Petersen-Wylie \cite{HPW} for quasi-Einstein manifolds. 

\begin{theorem}\label{complete'}
	Let $(M^n, g, f)$, $n\geq 5$, be an $n$-dimensional, non-trivial, complete  quasi-Einstein manifold satisfying \eqref{qe} with harmonic Weyl curvature and $m>1$ ($m\neq \infty$). Then, it is one of the following types:

	\begin{enumerate}
		\item[{\rm (1)}] $(M^n, g, f)$ is 
a quotient of some warped product quasi-Einstein manifold of the form
		\[
		\left(\mathbb{R} ,\, ds^2 \right) \times\, _h\left(N^{n-1}, \bar{g} \right), 
		\]
		where $\left(N^{n-1}, \bar{g} \right) $ is an Einstein manifold and $f=f(s)$ is a function on $\mathbb{R}$.
		
		\smallskip
		\item[{\rm (2)}] $\lambda<0$ and $(M^n,g)$ is isometric to a quotient of  the       
                        Riemananian product 
\[ \left(M_1^{k}, {g_1}\right)\times\left(M_2^{n-{k}}, {g_2}\right), \qquad 2\leq {k}\leq n-2, \] 
where  $\left(M_2^{n-{k}}, {g_2}\right) $ is an Einstein manifold with Einstein constant $\lambda$; 
$\left(M_1^{k}, {g_1}\right)$ $=\left(\mathbb{R}\times F^{{k}-1}, \ ds^{2}+ e^{2\sqrt{-\Lambda}s}g_{F}\right)$, for $\Lambda=\frac{1}{m+{k}-1}\lambda$ and $F^{{k}-1}$ ($3\leq {k}\leq n-2$) being Ricci-flat; and $f =-m{\sqrt{-\Lambda}s}$. Furthermore, $\left(M_1^{k}, {g_1}, f\right)$ is a (D-falt) quasi-Einstein manifold satisfying \eqref{qe} which is also an Einstein manifold with Einstein constant $\rho=\frac{{k}-1}{m+{k}-1}\lambda$. 

          \smallskip
          \item[{\rm (3)}] 
	$\lambda<0$ and $(M^n,g)$ is isometric to a quotient of  the Riemananian product 
\[ \left(\mathbb{H}^{{k}}, g_{\mathbb{H}}\right)\times\left(M_2^{n-{k}}, {g_2}\right), \qquad 2\leq {k}\leq n-2, \] 
where  $\left(M_2^{n-{k}}, {g_2}\right) $ is an Einstein manifold with Einstein constant $\lambda$;  $\left(\mathbb{H}^{{k}}, g_{\mathbb{H}}\right)$ is the hyperbolic space of constant sectional curvature $\Lambda=\frac{1}{m+{k}-1}\lambda$, with the metric $g_{\mathbb{H}}=ds^{2}+ \sqrt{-\Lambda} \sinh^2(\sqrt{-\Lambda} s) g_{\mathbb{S}^{k-1}}$; and $f= -m\log(\cosh \sqrt{-\Lambda}s)$.   Furthermore, $\left(\mathbb{H}^{{k}}, g_{\mathbb{H}}, f\right)$ is a (D-falt) quasi-Einstein manifold satisfying \eqref{qe}.  
\end{enumerate}
\end{theorem}	

\begin{remark} 
	In Theorem \ref{complete'}, quasi-Einstein manifolds of types (2) and (3) provide new examples of neither locally conformally flat nor $D$-flat ones. 
\end{remark}	

\begin{remark} 
Theorem \ref{complete'} is derived from a local classification theorem (Theorem~\ref{local}). 
Moreover, there is a version of Theorem \ref{complete'} in which we allow $M^n$ to have a boundary; see Section 5 for details. 
\end{remark}

\begin{remark} 
	By similar arguments used in this paper, one can also obtain a local classification of  $(m,\rho)$-quasi-Einstein metrics with harmonic Weyl curvature.
\end{remark}


\noindent {\bf Organization of the Paper.} 
In Section 2, we fix our notation and recall some basic facts from Riemannian geometry and quasi-Einstein manifolds. By adopting and extending the arguments used in \cite{BHS2025} for gradient Ricci solitons, we then derive the local multiply warped product structure of quasi-Einstein manifolds with harmonic Weyl curvature.
In Section 3, we show that the Ricci tensor has at most three mutually distinct eigenvalues and provide a local classification of quasi-Einstein manifolds with harmonic curvature when the Ricci tensor is very special. This latter case corresponds precisely  to the $D$-flat case. In Section 4, we analyze the remaining situations and characterize their local structures.  Finally, in Section 5, we state the local classification theorem (Theorem \ref{local}) for quasi-Einstein manifolds with harmonic Weyl curvature and discuss the global classifications.


\section{Preliminaries}

In this section, we fix our notation and recall some basic facts
and known results about Riemannian manifolds 
and general quasi-Einstein manifolds. Moreover, we describe some useful features of quasi-Einstein manifolds with harmonic Weyl curvature that will be used in later sections. 

\subsection{Notations and basics of  Riemannian geometry}

First of all, we recall that on any $n$-dimensional Riemannian
manifold $(M^n, g_{ij} )$ ($n\ge 3 $) the Weyl conformal curvature tensor is given by
\begin{equation*}\label{2.15}
W_{ijkl}:=R_{ijkl}-\frac{1}{n-2}
(A_{ik}g_{jl}+A_{jl}g_{ik}-
A_{il}g_{jk}-A_{jk}g_{il}),
\end{equation*}
where $A=A_{ij}$ is the Schouten tensor given by 
\begin{equation}\label{2.14}
	A_{ij}:=R_{ij}-\frac{1}{2(n-1)}R g_{ij}.
\end{equation}

As is well-known, the Weyl tensor $W$ is identically zero for every 3-manifold and, 
for $n\geq 4$, the vanishing of the Weyl tensor is equivalent to the locally
conformal flatness of $(M^n,g)$. We also recall that, in dimension $n=3$, 
$(M^3, g)$ is locally conformally flat if and only if its Cotton tensor $C$, defined by 
\begin{equation}\label{2.16}
C_{ijk}=\nabla_i A_{jk}-\nabla_j A_{ik},
\end{equation}
vanishes. Furthermore, for $n\geq 4$,  
the Cotton tensor can also be defined as a divergence of the Weyl tensor:
\begin{equation*}\label{2.17}
C_{ijk}=-\frac{n-2}{n-3}  \sum_l\nabla^l W_{ijkl}.
\end{equation*}

\subsection{Basic facts and identities for quasi-Einstein manifolds}  
\begin{lemma}  
	Let $(M^n, g, f)$ be a quasi-Einstein manifold satisfying \eqref{qe}, where $m\notin{\{\pm1,\,(2-n),\, \infty\}}$. Then, we have
	\begin{equation*}\label{2.21}
	R+ \Delta f-\frac 1 m |\nabla f|^2 =n\lambda,
	\end{equation*}
	\begin{equation*}\label{2.22}
\nabla R= \frac {2(m-1)} {m}Ric (\nabla f)+\frac 2 m [R-(n-1)\lambda] \nabla f
\end{equation*}
	and
	\begin{equation}\label{2.23}
\nabla  \left[R+ \frac {(m-1)} {m}|\nabla f|^2 -2\lambda  f \right]= \frac 2 m (|\nabla f|^2-\Delta f)\nabla f.
\end{equation}
\end{lemma}

 \begin{remark}
When $m=\infty$, the above two equations coincide with the well-known identities for  Ricci solitons: 
\[
\nabla R=2Ric(\nabla f) \quad \mbox {and}  \] 
\[ R+|\nabla f|^2-2\lambda f=C.
\]
 \end{remark}

\medskip
On any quasi-Einstein manifold  $(M,g, f)$, also known as an $(\lambda, n+m)$-Einstein manifold, we recall the $D$-tensor first introduced in \cite{CC12} (see also \cite{CC13}),  
\begin{equation*}\label{2.24}
D_{ijk} = \frac{1} {n-2} (A_{ij} \nabla_kf- A_{ik} \nabla_jf) + \frac {1} {(n-1)(n-2)} (g_{ij}E_{kl} - g_{ik}E_{jl})\nabla_lf,
\end{equation*}
where $E_{ij}=R_{ij}-\frac{R}{2}g_{ij}$ is the Einstein tensor.
This 3-tensor $D_{ijk}$ is closely tied to the Cotton tensor and the Weyl tensor, and played a significant role in \cite{CC12} and  \cite{CC13} on classifying locally conformally flat gradient steady Ricci solitons and Bach flat shrinking Ricci solitons.

\begin{lemma}[Chen-He \cite{CH}]
	Let $(M^n, g, f)$, $n\geq 4$, be an $n$-dimensional quasi-Einstein manifold satisfying \eqref{qe}, where $m\notin{\{\pm1,\,(2-n),\, \infty\}}$. Then, we have 
	\begin{equation}\label{2.26}
	\frac{m+n-2}{m}D_{ijk}=C_{ijk}+R_{ijkl}\nabla_l f.
	\end{equation}
\end{lemma}

\medskip

The following useful result is essentially parallel to that in gradient Ricci solitons with harmonic Weyl tensor, and can be easily obtained by using equations \eqref{2.14}, \eqref{2.16} and \eqref{2.23}; see, e.g., \cite{HW} or \cite{Shin}.  

\begin{lemma}\label{lemma2.3}
	Let $(M^n, g, f)$, $n\geq 4$, be an $n$-dimensional quasi-Einstein manifold satisfying \eqref{qe} with harmonic Weyl curvature, i.e.,  $\delta W = \nabla^l W_{ijkl} =0$, and $m\notin{\{\pm1,\,(2-n),\, \infty\}}$.
	Then,  the Cotton tensor $C$ vanishes and 
 the Schouten tensor $A$ is Codazzi.
	Moreover, we have
	\begin{equation}
	\begin{aligned}\label{2.27}
R_{ijkl}\nabla_l f=&	\nabla_i\nabla_j \nabla_k f -\nabla_j\nabla_i \nabla_k f\\
			=&-\nabla_i R_{jk}+\nabla_j R_{ik}+\frac{1}{m}(R_{jk}\nabla_i f-R_{ik}\nabla_jf)\\
         &- \frac{\lambda}{m} \left(g_{jk}\nabla_i f  - g_{ik}\nabla_jf \right)\\
	=&\frac{1}{2(n - 1)} \left( g_{ik}\nabla_jR - g_{jk}\nabla_iR \right)+\frac{1}{m}(R_{jk}\nabla_i f-R_{ik}\nabla_jf)\\
       &- \frac{\lambda}{m} \left(g_{jk}\nabla_i f  - g_{ik}\nabla_jf \right)\\
	=&\frac{m-1}{m(n - 1)} \nabla_lf \left( R_{lj}g_{ik} - R_{li} g_{jk}\right)+\frac{1}{m}(R_{jk}\nabla_i f-R_{ik}\nabla_jf)\\
          &+ \frac{R}{m(n - 1)} \left( g_{ik}\nabla_jf - g_{jk}\nabla_if\right).
	\end{aligned}
	\end{equation}
\end{lemma}

\subsection{Local multiply warped product structure of quasi-Einstein manifolds with harmonic Weyl curvature}

First of all, following a similar argument as in the proof of Lemma 3.3 in \cite{CC12}, we have the following basic properties that are valid for quasi-Einstein manifolds with harmonic Weyl curvature; see \cite[Lemma 1]{Shin} for more details.

\begin{lemma}
Let $(M^n, g, f )$, $n \geq 3$, be a quasi-Einstein manifold  with harmonic Weyl
curvature and $m\neq 2-n$, $c$ be a regular value of $f$,  and 
$\Sigma_c=\{ x|\ f(x))=c\} $
be the   level    surface  of   $f$. Then, the following hold:

\begin{enumerate}
	\item[{\rm (i)}] Whenever $\nabla f\neq 0$, $E_1 := \frac{\nabla f}{|\nabla f|}$ is an eigenvector field of $Ric$.

\smallskip
\item[{\rm (ii)}] 
The scalar curvature $R$ and $|\nabla f |^2$ are constant on a connected component of $\Sigma_c$.

\smallskip
\item[{\rm (iii)}] 
There is a function $s$, locally defined by $s=\int \frac{df} {|\nabla f|}$, so that $ds =  \frac{df}{|\nabla f|}$ and $E_1 = \nabla s$.

\smallskip
\item[{\rm (iv)}] 
The special eigenvalue, $\lambda_1:=Ric(E_1, E_1)$, of the Ricci tensor is constant on a connected component of $\Sigma_c$. 

\smallskip
\item[{\rm (v)}] 
 $\nabla_{E_1} E_1 = 0$.

\end{enumerate}
\end{lemma}

Next,  we explore the key feature that having harmonic Weyl curvature is equivalent to 
the Schouten tensor $A$ being a Codazzi tensor. In \cite[Lemma 2]{De}, Derdzi\'{n}ski showed the following useful property:  
for any Codazzi tensor $A$ and any point $x\in M$, let $E_A(x)$ be the number of distinct eigenvalues of $A_x$ (or $Ric$),
and set 
\[
M_A = \{  x \in M \ | \ E_A {\rm \ is \ constant \ in \ a \ neighborhood \ of \ } x \}.
\] 
Then $M_A$ is an open dense subset of $M$. Moreover, on each connected component of $M_A$, 
the eigenvalues of $A$ (or $Ric$)  are well-defined and differentiable functions; the corresponding eigenspaces of $A$ form mutually orthogonal differentiable distributions.

On the other hand, it is known that every quasi-Einstein manifold  $(M, g, f)$ is real analytic in harmonic coordinates; see \cite[Proposition 2.4]{HPW}. Thus, if $f$ is not a constant function, then the set of regular points $\{\nabla f \neq 0\}$ of $f$ is open and dense in $M$.  In particular, 
\[M_A \cap \{ \nabla f \neq 0  \}\subset M \] is an {open and dense} subset of  $M^n$. 

Hence, for each point $p\in M_A \cap \{ \nabla f \neq 0  \}$, there exists a neighborhood $U$ of $p$, 
such that the number of distinct eigenvalues of the Ricci tensor is constant on $U$. Suppose that, aside from $\lambda_1:=Ric(E_1, E_1)$, there are $m'$ distinct Ricci eigenvalues with multiplicities are $r_{1}, r_{2}, \cdots, r_{m'}$, respectively, such that $1+r_{1}+r_{2}+ \cdots+ r_{m'}=n $. Then, we can choose a local frame $\{E_1= \frac{\nabla f}{|\nabla f| }, E_2, \cdots,  E_n\}$ over $U$,
such that 
\[
R_{ij}=\lambda_i g_{ij}.
\]
Without loss of generality, we may assume that 
\[
\lambda_{2}=\cdots=\lambda_{r_1+1}, \
\lambda_{r_1+2}=\cdots=\lambda_{r_1+r_2+1},\
\cdots,  \  
\lambda_{r_1+r_2+\cdots+r_{m'-1}+2}=\cdots=\lambda_{n},
\]
and that $\lambda_{2},\, \lambda_{r_1+2},\,\cdots,\,	\lambda_{r_1+r_2+\cdots+r_{m'-1}+2}$
are distinct.

\begin{lemma} 
	Let $(M^n, g, f)$, $n\geq 4$, be an $n$-dimensional quasi-Einstein manifold satisfying \eqref{qe} with harmonic Weyl curvature and  $m\notin{\{\pm1,\,(2-n),\, \infty\}}$. 
	Let $\{E_1=\nabla f/|\nabla f|, E_2, \cdots, E_n\}$  be a frame consisting of eigenvector fields of the Ricci tensor $Ric$ in some neighborhood $U\subset M_A\cap\{\nabla f\neq 0\}$ of a regular level surface $\Sigma_c:=\{f=c\}$. Then, for each $1\le i\le n$,  the eigenvalue $\lambda_i$ is constant on any connected components of $\Sigma_c$, hence a function of $s$ only. 
\end{lemma}

\begin{proof} This is an extension of \cite[Lemma 3]{Shin}, by using essentially the same argument as in the proof of \cite[Lemma 2.5]{Kim2} and \cite[Lemma 3.3]{Li}. 	
\end{proof}

With these preparations done, by adopting and extending the arguments used in \cite{Li, BHS2025}, we proceed to obtain the local multiply warped product structure of $n$-dimensional quasi-Einstein manifold satisfying \eqref{qe} with harmonic Weyl curvature.

In the following and the rest of the paper, we shall use the following convention: for $2\leq a, \alpha\leq n$, let us denote by 
\[[a]=\{b \ \!| \ 2\leq b\leq n ~ {\rm and } ~ \lambda_{b}=\lambda_{a}\} \subset  \{2, \cdots, n\}\] 
and 
\[[\alpha]=\{\beta \ \!| \ 2\leq \beta \leq n ~ {\rm and } ~ \lambda_{\beta}=\lambda_{\alpha} \} \subset  \{2, \cdots, n\}\] 
such that $[a]\neq[\alpha]$. In particular, $\lambda_a$ and $\lambda_{\alpha}$ will always have distinct values. 

First, we state a very useful lemma from \cite[Lemma 7]{BHS2025} which is also valid for quasi-Einstein manifolds with harmonic Weyl curvature (with essentially the same proof). 

\begin{lemma}\label{bhs}
Let $(M^n, g, f)$, $n\geq 4$, be an $n$-dimensional quasi-Einstein manifold satisfying \eqref{qe} with harmonic Weyl curvature and $m\notin{\{\pm1,\,(2-n),\,\infty\}}$.  On $M_A \cap \{ \nabla f \neq 0  \}$, let $(x_{b})_{b\in[a]}$ and $(x_{\beta})_{\beta\in[\alpha]}$ be any local coordinate systems of the integral manifolds of the eigenspaces with corresponding Ricci eigenvalues  $\lambda_{a}$ and $\lambda_{\alpha}$, respectively.  Then, 
	setting $\partial_{1}:=E_{1}=\nabla f/|\nabla f|$, we have
	\begin{equation*}
		\partial_{1}g_{ab}=2\xi_{a}g_{ab}\ \ \ \ \text{and}\ \ \ \ \ \partial_\alpha g_{ab}=0,
	\end{equation*}
 where $\xi_{a}=\frac{\lambda-\lambda_{a}}{|\nabla f|}$.
\end{lemma}

Now, for each $p \in M_A \cap \{ \nabla f \neq 0  \}$, we assume that, without loss of generality, the multiplicities of the distinct Ricci eigenvalues are given by \[ r_{1}=r_{2}=\cdots=r_{l}=1 \quad \& \quad r_{l+1},~ r_{l+2},~\cdots,~ r_{m'} \geq2\]
for some $0\leq l\leq m'$. 

By using essentially the same method as in dealing with the Ricci soliton case (see  \cite[Theorem 3]{BHS2025} and \cite[Theorem 3.6]{Li}),  we can derive the following local multiply warped product structure, as well as the associated integrability conditions.  

\begin{theorem} \label{mulwar}
	Let $(M^n, g, f)$, $n\geq 4$, be an $n$-dimensional quasi-Einstein manifold satisfying \eqref{qe} with harmonic Weyl curvature and $m\notin{\{\pm1,\,(2-n),\, \infty\}}$.
	Then, for each point $p \in M_A \cap \{ \nabla f \neq 0  \}$, there exists a neighborhood 
$U\subset M_A \cap \{ \nabla f \neq 0  \}$ of $p$ 
	such that, over $U$, $g$ can be expressed as the following multi-warped product metric:
	\begin{equation} \label{3.10}
	g_{|U}= ds^2 + h^2_1(s)  dt_1^2 + \cdots 
	+ h^2_l(s) dt_l^2+ h^2_{l+1}(s) \tilde{g}_{l+1}+ \cdots 
	+h^2_{m'}(s) \tilde{g}_{m'}.
	\end{equation}
Here, each $h_j(s)$, $1\leq j \leq m$, is a smooth positive function in $s$;
each $\tilde{g}_{\mu}$,  $l+1\leq \mu \leq m'$, 
is an  Einstein metric, with Einstein constant $(r_\mu-1)k_\mu$, over an $r_\mu$-dimensional Einstein manifold. 
	
	Moreover, let $\lambda_1, \cdots, \lambda_n$ be the Ricci-eigenvalues, with $\lambda_1$ being the Ricci-eigenvalue with respect to the gradient vector $\nabla f$, 
	and $\xi_a:= \frac{1}{ |\nabla f|} (\lambda-\lambda_{a})$, $2\leq a\leq n$.  
	Then, the following integrability conditions hold
	\begin{equation}\label{3.11}
	\xi'_a+\xi^2_a+\frac{1}{m}f'\xi_a=-\frac{R'}{2(n-1)f'}, 
	\end{equation}
	\begin{equation}\label{3.12}
	\lambda'_a-\left(\lambda_1-\lambda_a \right)\xi_a=\frac{R'}{2(n-1)},
	\end{equation}
	\begin{equation}\label{3.13}
	\lambda_{1}=-f''+\frac{1}{m}f'^2+\lambda =-\sum^n_{i=2}\left( \xi'_i+\xi^2_i\right),
	\end{equation}
	and
	\begin{equation}\label{3.14}
	\lambda_a=-f'\xi_a+\lambda=- \xi'_a-\xi_a \sum^{n}_{i=2} \xi_i +( r-1)\frac{k}{h^2}.
	\end{equation}
	Here, $h$ is determined by $\xi_a=h'/h$; in addition, in \eqref{3.14}, we have $r=1$ and  $h=h_{a-1}$ for $2\leq a \leq l+1$, and $r=r_\mu$, $h=h_{\mu} $ and $k=k_{\mu} $ for $a\in \left[ l+r_{l+1}+\cdots+r_{\mu-1} +2\right]$, where we have used the convention 
	$[a]=\{b| \ \lambda_{b}=\lambda_{a}~ {\rm and } ~ b \neq1\}$, $2\leq a\leq n$, as before. 
\end{theorem}

\begin{remark} Note that $\xi_a=h'/h$ follows from \eqref{3.10}. Also, formula \eqref{3.11} may look different from the corresponding one, $ \xi_a^{'}   -  \xi_\alpha^{'} =  -  ( \xi_a^2 -  \xi_\alpha^2)$, in  the Ricci soliton case \cite{Li}. But the similar crucial formulas \eqref{4.11} and \eqref{5.7} in later sections are still valid.
\end{remark}

\begin{remark}
	For \eqref{3.14}, note that the term $( r-1)\frac{k}{h^2}$ drops when $r=1$. However, since it vanishes for any constant $k$, sometimes we still keep it just for convenience. 
\end{remark}

\section{The local structure of the case with more than two distinct Ricci-eigenfunctions}

The classifications of an $n$-dimensional quasi-Einstein manifold $(M^n, g, f)$ with harmonic Weyl curvature, as stated in Theorem \ref{complete} (and Theorem \ref{local}) will be proved according to how many mutually distinct Ricci-eigenvalues, as well as their multiplicities, it has. 
In this section, we shall investigate whether the Ric eigenvalues\footnote{To avoid repetition, unless stated otherwise, the Ricci-eigenvalues mentioned in the following discussions do not include $\lambda_1$, the eigenvalue of Ric with respect to $\nabla f$. Also, we denote the other Ricci-eigenvalues by $\lambda_a$, $2\leq a \leq n$.} $\lambda_2,~\lambda_3,~\cdots,~\lambda_n$ can have three or more mutually different values. As we shall see, this case cannot occur. Moreover, at the end of the section, we shall also treat the case that all Ricci-eigenfunctions, other than $\lambda_1$, are equal (i.e., $\lambda_2=\cdots =\lambda_ n$).

\medskip	
First of all, we analyze the integrability conditions in Theorem \ref{mulwar}.
Assume that $\lambda_a \neq \lambda_\alpha$ are two distinct Ricci eigenvalues, 
of multiplicities $r_1$ and $r_2$, respectively.
Set 
\begin{equation} \label{X&Y}
X:=\xi_a = \frac{1}{ |\nabla f|} (\lambda-\lambda_{a}) \qquad {\rm and}  \qquad Y:=\xi_\alpha= \frac{1}{ |\nabla f|} (\lambda-\lambda_{a}).
\end{equation}
Then, from \eqref{3.11}-\eqref{3.14}, we have 
\begin{equation}\label{4.1}
\begin{aligned}
X'+X^2+\frac{1}{m}f'X & =Y'+Y^2+\frac{1}{m}f'Y\\
  & =\xi'_i+\xi^2_i+\frac{1}{m}f'\xi_i=- \frac{R'}{2(n-1)f'},
\end{aligned}
\end{equation}
\begin{equation}\label{4.2}
	\begin{aligned}
\lambda_{1}&=-f''+\frac{1}{m}f'^2+\lambda
=-\sum^n_{i=2}\left( \xi'_i+\xi^2_i\right)\\
&=(n-1)\frac{R'}{2(n-1)f'}+\frac{1}{m}f'\sum^n_{i=2}\xi_i\\
&=-(n-1)\left( X'+X^2 +\frac{1}{m}f'X\right)
+\frac{1}{m}f'\sum^n_{i=2}\xi_i,
	\end{aligned}
\end{equation}
\begin{equation}\label{4.3}
\begin{aligned}
\lambda_{a} & =-f'X+\lambda\\
 & =-\left( X'+X^2+\frac{1}{m}f'X\right)+\frac{1}{m}f'X+X^2+(r_1-1)\frac{k_1}{h^2_1} 
-X\sum^n_{i=2}\xi_i,
\end{aligned}
\end{equation}
\begin{equation}\label{4.4}
\begin{aligned}
\lambda_{\alpha} & =-f'Y+\lambda\\
& =-\left(Y'+Y^2+\frac{1}{m}f'Y\right)+\frac{1}{m}f'Y+Y^2+(r_2-1)\frac{k_2}{h^2_2}
-Y\sum^n_{i=2}\xi_i
\end{aligned}
\end{equation}
and 
\begin{equation}\label{4.5}
\lambda'_a-\left(\lambda_1-\lambda_a \right)X=\lambda'_\alpha-\left(\lambda_1-\lambda_\alpha \right)Y.
\end{equation}

By using the above basic facts,  we have the following 
\begin{lemma}\label{lemma4.1}
Let $(M^n, g, f)$, $n\geq 4$, be an $n$-dimensional quasi-Einstein manifold satisfying \eqref{qe} with harmonic Weyl curvature and  $m\notin{\{\pm1,\,(2-n),\, \infty\}}$. Suppose,  
in some neighborhood $U$ of $p\in M_{\rm Ric} \cap \{ \nabla f \neq 0  \}$, $\lambda_a$ and $\lambda_\alpha$ are distinct Ricci-eigenvalues with multiplicities $r_1$ and $r_2$, respectively. 
Then, the following identities hold:	
	
	\begin{equation}\label{4.6}
	(r_1-1)\frac{k_1}{h^2_1}-(r_2-1)\frac{k_2}{h^2_2}=(X-Y)\left[ \sum^n_{i=2}\xi_i-(X+Y)-\frac{m+1}{m}f' \right],
	\end{equation}

       \begin{equation}\label{4.7}
	\begin{aligned}
	(r_1-1)\frac{k_1}{h^2_1}+ (r_2-1)\frac{k_2}{h^2_2} =&\frac{2m+n-1}{m}(X'+X^2+\frac{1}{m}f'X )+\frac{m+1}{m}\lambda \\
& +\left[ \sum^n_{i=2}\xi^2_i-(X^2+Y^2) \right] -\frac{m+1}{m^2}f'\sum^n_{i=2}\xi_i\\
	&+\frac{2}{m}f'\left[\sum^n_{i=2}\xi_i-(X+Y) \right],
	\end{aligned}
	\end{equation}

      \begin{equation}\label{4.8}
	\begin{aligned}
	(r_1-1)\frac{k_1}{h^2_1}X- (r_2-1)\frac{k_2}{h^2_2}Y =& (X-Y)\left[ (X'+X^2+\frac{1}{m}f'X) +\lambda\right] \\
&+(X-Y)(X+Y)\left[\sum^n_{i=2}\xi_i-(X+Y) \right]\\
	&+(X-Y)\left[ XY-\frac{m+1}{m} f'(X+Y)\right],
	\end{aligned}
	\end{equation}

	\begin{equation}\label{4.9}
	-(r_1-1)\frac{k_1}{h^2_1}Y+(r_2-1)\frac{k_2}{h^2_2}X
	=(X-Y)\left[(X'+X^2 +\frac{1}{m}f'X ) +\lambda+XY\right],
	\end{equation}
	
\begin{equation}\label{4.10}
	\begin{aligned}
	&\frac{m-1}{m}\left[ (r_1-1)\frac{k_1}{h^2_1}X- (r_2-1)\frac{k_2}{h^2_2}Y  \right]\\
	&=(X-Y)\left\{ \frac{m+n-2}{m}(X'+X^2+\frac{1}{m}f'X  )+\left[ \sum^n_{i=2}\xi^2_i-(X^2+Y^2)  \right]-\frac{m+1}{m}XY\right\}\\
	&\ \ +(X-Y)\left\{\frac{1}{m}f'\left[\sum^n_{i=2}\xi_i-(X+Y)\right ]-\frac{1}{m}(X+Y+\frac{1}{m}f')[\sum^n_{i=2}\xi_i-(X+Y) ]\right\}
	\end{aligned}
	\end{equation}
	and 
	\begin{equation}\label{4.11}
	\begin{aligned}
	\frac{n-1}{m}(X'+X^2+\frac{1}{m}f'X )  - &\frac{m-1}{m}\lambda+\sum^n_{i=2}\xi^2_i+\frac{m-1}{m^2}f'\sum^n_{i=2}\xi_i\\
	=&(X+Y)\left( \sum^n_{i=2}\xi_i -\frac{m-1}{m}f' \right).
	\end{aligned}
	\end{equation}
\end{lemma}

\begin{proof}
	First, subtracting \eqref{4.3} from \eqref{4.4} gives us 
	\begin{equation}
	\begin{aligned}\label{4.12}
	\lambda_{\alpha}-\lambda_a=f'(X-Y)
	=&-\frac{1}{m}f'(X-Y)+(X-Y)\left[ \sum^n_{i=2}\xi_i-(X+Y) \right]\\
	& -\left[(r_1-1)\frac{k_1}{h^2_1}-(r_2-1)\frac{k_2}{h^2_2}\right],
	\end{aligned}
	\end{equation}
	which implies that 
	\[\frac{m+1}{m}f'(X-Y)=(X-Y)\left[ \sum^n_{i=2}\xi_i-(X+Y)\right]
	-\left[(r_1-1)\frac{k_1}{h^2_1}-(r_2-1)\frac{k_2}{h^2_2}\right],
	\]
	yielding \eqref{4.6}.
	Moreover, differentiating the above equation  leads to 
	\begin{equation}\label{4.13}
	\begin{aligned}
	\frac{m+1}{m}\left[f'(X-Y)\right]' =& (n-3) (X-Y)(X'+X^2+\frac{1}{m}f'X ) \\
& - (X-Y)(X+Y+\frac{1}{m}f')\left[ \sum^n_{i=2}\xi_i-(X+Y) \right]\\
	&-(X-Y)\left[ \sum^n_{i=2}\xi^2_i-(X^2+Y^2)\right] \\
& -(X-Y)\frac{1}{m}f'\left[ \sum^n_{i=2}\xi_i-(X+Y) \right]\\
	&+2\left[ (r_1-1)\frac{k_1}{h^2_1}X-(r_2-1)\frac{k_2}{h^2_2}Y\right],
	\end{aligned}
	\end{equation}
  where we have used $h'_1/{h_1}=X$ and $h'_2/{h_2}=Y$.
	On the other hand, by applying identities \eqref{4.1}, \eqref{4.2} and \eqref{4.12}, we see that
	\begin{equation}\label{4.14}
	\begin{aligned}
	\left[f'(X-Y)\right]'=& f''(X-Y)-f'(X-Y)(X+Y+\frac{1}{m}f')\\
	=&(X-Y)\left\{(n-1)(X'+X^2+\frac{1}{m}f'X ) +\lambda-\frac{1}{m}f'\sum^n_{i=2}\xi_i +\frac{1}{m}f'^2\right\}\\
	&-\frac{m}{m+1}(X-Y)\left[ \sum^n_{i=2}\xi_i-(X+Y) \right] (X+Y+\frac{1}{m}f')\\
	&+\frac{m}{m+1}(X+Y+\frac{1}{m}f')\left[ (r_1-1)\frac{k_1}{h^2_1}-(r_2-1)\frac{k_2}{h^2_2}\right].
	\end{aligned}
	\end{equation}
	Comparing \eqref{4.13} with \eqref{4.14}, and using the fact $X\neq Y$, gives us \eqref{4.7}. 
	
	\smallskip
	Next, we take the harmonic Weyl condition into the consideration. It follows from identities \eqref{4.2}, \eqref{4.3} and \eqref{4.4} that 
	\begin{equation}
	\begin{aligned}\label{4.15}
	(\lambda_1-\lambda_\alpha)Y -(\lambda_1-\lambda_a)X & =(n-2) (X-Y)(X'+X^2 +\frac{1}{m}f'X)\\
 & - (X-Y) (X+Y+\frac{1}{m}f')\sum^n_{i=2}\xi_i \\
	&+(X-Y)\left[X^2+XY+Y^2+\frac{1}{m}f'(X+Y)\right]\\
	&+\left[ (r_1-1)\frac{k_1}{h^2_1}X-(r_2-1)\frac{k_2}{h^2_2}Y\right].
	\end{aligned}
	\end{equation}
	Since $\lambda'_\alpha-\lambda'_a=\left(\lambda_1-\lambda_\alpha \right)Y-\left(\lambda_1-\lambda_a \right)X$ by \eqref{4.5},
	comparing \eqref{4.14} and \eqref{4.15} shows \eqref{4.8}. In addition, 
	substituting
		\begin{equation*}
	\begin{aligned}	
	\frac{m+1}{m} & f'(X-Y)(X+Y)\\
	=&\left\{(X-Y)\left[ \sum^n_{i=2}\xi_i-(X+Y) \right]
	-\left[(r_1-1)\frac{k_1}{h^2_1}-(r_2-1)\frac{k_2}{h^2_2}\right]\right\}(X+Y)
	\end{aligned}
	\end{equation*}
	into \eqref{4.8} implies \eqref{4.9}.
	
Meanwhile, we obtain \eqref{4.10} by combining \eqref{4.13} and \eqref{4.15}.
Finally,	subtracting \eqref{4.8} from \eqref{4.10}  leads to \eqref{4.11}.  This completes the proof of the lemma.
\end{proof}	

Now, we are ready to apply equations \eqref{4.1}, \eqref{4.2} and \eqref{4.11} to obtain the following desired result.		
\begin{theorem}\label{thm4.1}
Let $(M^n, g, f)$, $n\geq 4$, be an $n$-dimensional quasi-Einstein manifold satisfying \eqref{qe} with harmonic Weyl curvature, where $m>1$ and $m\neq\infty$. Then, for each 
$p \in M_A \cap \{ \nabla f \neq 0  \}$,  there exists a neighborhood $U$ of $p$ on which 
the Ricci eigenvalues $\lambda_2,~\lambda_3,~\cdots,~\lambda_n$ can have at most two distinct values. 
\end{theorem}

\begin{proof} We prove by contradiction. Suppose not, then $\{\lambda_2,~\lambda_3,~\cdots,~\lambda_n\}$ have at least three distinct values, which we denote by $\lambda_a$, $\lambda_\alpha$ and $\lambda_q$,  with multiplicities $r_1$, $r_2$ and $r_3$, respectively.
	For convenience, we also denote $X=\xi_a$, $Y=\xi_\alpha$ as in \eqref{X&Y}, and  $Z=\xi_{q}$.	By \eqref{4.11}, 
	and under the assumption of Lemma \ref{lemma4.1}, we see that
	\begin{equation*}
	\begin{aligned}
	\frac{n-1}{m}(X' &+X^2+\frac{1}{m}f'X ) -\frac{m-1}{m}\lambda+\sum^n_{i=2}\xi^2_i+\frac{m-1}{m^2}f'\sum^n_{i=2}\xi_i\\
	=&(X+Y)\left( \sum^n_{i=2}\xi_i -\frac{m-1}{m}f' \right)\\
	=&(X+Z)\left( \sum^n_{i=2}\xi_i -\frac{m-1}{m}f' \right)\\
	=&(Y+Z)\left( \sum^n_{i=2}\xi_i -\frac{m-1}{m}f' \right).
	\end{aligned}
	\end{equation*}
	Therefore, since $X$, $Y$ and $Z$ are mutually distinct, it follows that
	\begin{equation}\label{4.16}
	\sum^n_{i=2}\xi_i =\frac{m-1}{m}f', 
	\end{equation}
	 and then we have
	\begin{equation}\label{4.17}
	\frac{n-1}{m}(X'+X^2+\frac{1}{m}f'X ) -\frac{m-1}{m}\lambda+\sum^n_{i=2}\xi^2_i+\frac{m-1}{m^2}f'\sum^n_{i=2}\xi_i=0.
	\end{equation}
	Differentiating \eqref{4.16}, with \eqref{4.1} and \eqref{4.2}, we obtain
	\begin{equation}\label{4.18}
	\frac{n-1}{m}(X'+X^2+\frac{1}{m}f'X ) -\frac{m-1}{m}\lambda-\sum^n_{i=2}\xi^2_i-\frac{1}{m^2}f'\sum^n_{i=2}\xi_i-\frac{m-1}{m^2}f'^2=0.
	\end{equation}
	By comparing \eqref{4.18} with \eqref{4.17}, and using \eqref{4.16}, it follows that
	\[
	\sum^n_{i=2}\xi^2_i+\frac{m-1}{m^2}f'^2=0.
	\]
	Hence,  $\sum^n_{i=2}\xi^2_i=\frac{m-1}{m^2}f'^2=0$ when $m>1$. This is a contradiction,  
	since $X,~Y ~{\rm and }~Z$ are mutually distinct. 
\end{proof}

\begin{remark}
   From \eqref{4.9}, it follows easily that there are no more than two (mutually) distinct Ricci eigenvalues, among $\lambda_2,~\lambda_3,~\cdots,~\lambda_n$, of multiplicity one (without requiring $m>1$).   
  Indeed, suppose $r_1= r_2 =r_3=1$. Then,  by \eqref{4.9}, we have 
   \[
    (X'+X^2 +\frac{1}{m}f'X ) +\lambda+XY=0= (X'+X^2 +\frac{1}{m}f'X ) +\lambda+XZ,
   \]
which implies that $X(Y-Z)=0$. By the real analyticity of quasi-Einstein maniolds and $Y\neq Z$, we have $X=0$. Similarly, $Y=Z=0$, but that is impossible.
Therefore, we see that the case of $\lambda_2,~\lambda_3$ and $\lambda_4$ being mutually distinct cannot occur, which was first proved by Shin in \cite{Shin} when $n=4$.
\end{remark}

We conclude this section by treating the case that all Ricci-eigenfunctions, other than $\lambda_1$, are equal (i.e., $\lambda_2=\cdots =\lambda_ n$). This corresponds to quasi-Einstein maniolds of Type {\rm (a)} in Theorem \ref{complete}. 

\begin{theorem} \label{same}
	Let $(M^n, g, f)$, $n\geq 4$, be an $n$-dimensional quasi-Einstein manifold satisfying \eqref{qe} with harmonic Weyl curvature and $m\notin{\{\pm1,\,(2-n),\, \infty\}}$.	
	Suppose that all Ricci-eigenfunctions $\{\lambda_a\}$, $2\le a\le  n$, are equal.	Then, either $(M^n, g)$ is Einstein or the metric $g$ is locally a warped product of the form
	\begin{equation*} \label{metr}
	g= ds^2 +  h(s)^2 \tilde{g},
	\end{equation*}
	for some positive function $h$,
	where the Riemannian metric $\tilde{g}$ is Einstein. 
	Furthermore, the $D$-tensor of  $(M^n, g,f)$ vanishes. 
\end{theorem}

\begin{proof}
	Clearly, either we have all Ricci-eigenfunctions are equal, i.e., $\lambda_1=\lambda_2=\cdots=\lambda_n$, hence $(M^n, g)$ is Einstein or, by Theorem \ref{mulwar}, the metric $g$ is locally a warped product of the given form.		
	Furthermore, since the Cotton tensor $C_{ijk}=0$,
	it is easy to check that the vanishing of $D$-tensor 
follows from \eqref{2.26}; see also \cite{Li} for more details.
\end{proof}

\section{The local structure of the case with two distinct Ricci-eigenfunctions}

	In this section, we study the remaining case when there are exactly two distinct Ricci eigenvalues among
$\lambda_2, \cdots, \lambda_n$, yielding quasi-Einstein manifolds of types \!{\rm (\!i)}--{\rm (\!iv)} in Theorem \ref{two1} and of types \!{\rm ( \!i) \& {\rm ( \!ii) in Theorem \ref{two2}, respectively.

 Throughout this section, we denote by $\lambda_a$ and $\lambda_\alpha$ the two distinct values among $\lambda_2, \cdots, \lambda_n$, and assume they have multiplicities $r_1$ and $r_2:=n-r_1-1$, respectively. Again, we denote $X=\xi_a$ and $Y=\xi_\alpha$ as in \eqref{X&Y}. Then,  it is clear that equations \eqref{4.6}-\eqref{4.11} in Lemma \ref{lemma4.1} can be expressed as follows:
\begin{equation}\label{5.1}
(r_1-1)\frac{k_1}{h^2_1}-(r_2-1)\frac{k_2}{h^2_2}=(X-Y)\left[ (r_1-1)X+(r_2-1)Y-\frac{m+1}{m}f' \right],
\end{equation}
\begin{equation}\label{5.2}
\begin{aligned}
(r_1-1)\frac{k_1}{h^2_1}+ (r_2-1) \frac{k_2}{h^2_2} =& \frac{2m+n-1}{m}(X'+X^2+\frac{1}{m}f'X )+\frac{m+1}{m}\lambda\\
& +\left[ (r_1-1)X^2+(r_2-1)Y^2 \right]   -\frac{m+1}{m^2}f'(r_1X+r_2Y)\\ 
&+\frac{2}{m}f'[(r_1-1)X+(r_2-1)Y],
\end{aligned}
\end{equation}
\begin{equation}\label{5.3}
\begin{aligned}
(r_1-1)\frac{k_1}{h^2_1}X- (r_2-1)\frac{k_2}{h^2_2}Y 
= & (X-Y)\left[ (X'+X^2+\frac{1}{m}f'X ) +\lambda \right]
\\ 
&+(X-Y)(X+Y) \left[ (r_1-1)X+(r_2-1)Y \right]\\
&+(X-Y)\left[ XY-\frac{m+1}{m} f'(X+Y)\right],
\end{aligned}
\end{equation}
\begin{equation}\label{5.4}
-(r_1-1)\frac{k_1}{h^2_1}Y+(r_2-1)\frac{k_2}{h^2_2}X
=(X-Y)\left[\left((X'+X^2+\frac{1}{m}f'X \right)+\lambda+XY\right],
\end{equation}
\begin{equation}\label{5.5}
\begin{aligned}
\frac{m-1}{m}\left[ (r_1-1)\frac{k_1}{h^2_1}X   - (r_2-1)\frac{k_2}{h^2_2}Y  \right]= & \frac{m+n-2}{m}(X-Y)(X'+X^2+\frac{1}{m}f'X)\\
& +\frac{m-1}{m}(X-Y)\left[ (r_1-1)X^2+(r_2-1)Y^2\right]\\
&+\frac{m-1}{m^2} (X-Y) f'\left[(r_1-1)X+(r_2-1)Y\right ] \\
& -\frac{m+n-2}{m}(X-Y)XY, 
\end{aligned}
\end{equation}
and 
\begin{equation}\label{5.6}
\begin{aligned}
\frac{n-1}{m}(X'+X^2+\frac{1}{m}f'X ) -\frac{m-1}{m}\lambda= & (n-1)XY\\
& -\frac{m-1}{m^2}f' \left[(r_1-1)X+(r_2-1)Y\right ]\\
& -\frac{(m-1)(m+1)}{m^2}f' (X+Y).
\end{aligned}
\end{equation}

First of all, we shall need the following useful lemma. 

\begin{lemma}\label{lemma5.1}
	Let $(M^n, g, f)$, $n\geq 4$, be an $n$-dimensional quasi-Einstein manifold satisfying \eqref{qe} with harmonic Weyl curvature and $m\notin{\{\pm1,\,(2-n),\, \infty\}}$.
	In a neighborhood $U$ of $p$ in $M_A \cap \{ \nabla f \neq 0  \}$, 
	assume that the Ricci eigenvalues 
	$\lambda_2, \cdots, \lambda_n$  have exactly two distinct values, 
	denoted by $\lambda_a$ and $\lambda_\alpha$, of multiplicities $r_1$ and $r_2:=n-r_1-1$.
	Then, we have
	\begin{equation}\label{5.7}
	X'+X^2 +\frac{1}{m}f'X+XY=0
	\end{equation}
	and 
	\begin{equation}\label{5.8}
	\begin{aligned}
	(m+1)XY & -\frac{m-1}{m(n-1)}f'[(r_1-1)X+(r_2-1)Y]\\
	&-\frac{m^2-1}{m(n-1)}f'(X+Y)+\frac{m-1}{n-1}\lambda=0.
	\end{aligned}
	\end{equation}
\end{lemma}

\begin{proof}
	By \eqref{5.6}, \eqref{5.2} can be rewritten as
	\begin{equation}\label{5.9}
	\begin{aligned}
	(r_1-1)\frac{k_1}{h^2_1} &+   (r_2-1)\frac{k_2}{h^2_2}\\
	=& \ 2 (X'+X^2+\frac{1}{m}f'X)+[ (r_1-1)X^2+(r_2-1)Y^2 ]\\
	&+(n-1)XY-\frac{m+1}{m}f'(X+Y)+2\lambda\\
	=& \ 2 (X'+X^2+\frac{1}{m}f'X+XY+\lambda)\\
	&+\left[ (r_1-1)X+(r_2-1)Y-\frac{m+1}{m}f'\right](X+Y).
	\end{aligned}
	\end{equation}
	By using \eqref{5.6} again, we also have
	\begin{equation}\label{5.10}
	\begin{aligned}
	(r_1-1)\frac{k_1}{h^2_1} & + (r_2-1)\frac{k_2}{h^2_2} \\
=& \ 2(m+1)XY  +[ (r_1-1)X+(r_2-1)Y](X+Y) \\
	&-\frac{2(m-1)}{m(n-1)}f'[(r_1-1)X+(r_2-1)Y] \\
	& -\frac{(m+1)(2m+n-3)}{m(n-1)}f'(X+Y) +\frac{2(m+n-2)}{n-1}\lambda.
	\end{aligned}
	\end{equation}
Then, \eqref{5.1} and \eqref{5.9} imply that 	
	\begin{equation}\label{5.11}
	\begin{aligned}
	(r_1-1)\frac{k_1}{h^2_1}=&X\left[ (r_1-1)X+(r_2-1)Y-\frac{m+1}{m}f'\right]\\
	&+(X'+X^2+\frac{1}{m}f'X+XY+\lambda)
	\end{aligned}
	\end{equation}
and 
	\begin{equation}\label{5.12}
	\begin{aligned}
	(r_2-1)\frac{k_2}{h^2_2}=&Y\left[(r_1-1)X+(r_2-1)Y-\frac{m+1}{m}f'\right]\\
	&+(X'+X^2+\frac{1}{m}f'X+XY+\lambda)
	\end{aligned}
	\end{equation}
	Hence,
	\begin{equation}\label{5.13}
	\begin{aligned}
	(r_1-1)\frac{k_1}{h^2_1}X &+(r_2-1)\frac{k_2}{h^2_2}Y\\
	=&\left[(r_1-1)X+(r_2-1)Y-\frac{m+1}{m}f'\right](X^2+Y^2)\\
	&+\left(X'+X^2+\frac{1}{m}f'X+XY+\lambda\right)(X+Y)\\
	=&\left[(r_1-1)X+(r_2-1)Y-\frac{m+1}{m}f'\right](X+Y)^2\\
	&-2XY\left[(r_1-1)X+(r_2-1)Y-\frac{m+1}{m}f'\right]\\
	&+(X'+X^2+\frac{1}{m}f'X+XY+\lambda)(X+Y).
	\end{aligned}
	\end{equation}

To prepare for the next step in differentiating (4.10), let us first list some basic identities:  
	\begin{equation}\label{5.14}
	(X-Y)'=-(X-Y)\left( X+Y+\frac{1}{m}f'\right),
	\end{equation}
	\begin{equation}\label{5.15}
	(X+Y)'=2\left(X'+X^2+\frac{1}{m}f'X+XY\right) 
	-(X+Y)\left( X+Y+\frac{1}{m}f'\right),
	\end{equation}
	\begin{equation}\label{5.16}
	(XY)'=\left(X'+X^2+\frac{1}{m}f'X+XY\right)(X+Y) 
	-2XY\left( X+Y+\frac{1}{m}f'\right),
	\end{equation}
	\begin{equation}\label{5.17}
	\begin{aligned}
	\left[(r_1-1)X+ (r_2-1)Y\right]'=&(n-3)\left(X'+X^2+\frac{1}{m}f'X+XY\right)(X+Y)\\
	&- [(r_1-1)X+(r_2-1)Y]\left( X+Y+\frac{1}{m}f'\right),
	\end{aligned}
	\end{equation}
	\begin{equation}\label{5.18}
	\begin{aligned}
	f''=&(n-1)\left(X'+X^2+\frac{1}{m}f'X+XY\right)-\frac{1}{m}f'[(r_1-1)X+(r_2-1)Y]\\
	&-(n-1)XY-\frac{1}{m}f'(X+Y)+\frac{1}{m}f'^2+\lambda,
	\end{aligned}
	\end{equation}
	and
	\begin{equation}\label{5.19}
	\begin{aligned}
	\left(X'+X^2+\frac{1}{m}f'X+XY\right) &= (m+1)XY \\
  & -\frac{m-1}{m(n-1)}f'[(r_1-1)X+(r_2-1)Y]\\
	&-\frac{m^2-1}{m(n-1)}f'(X+Y)+\frac{m-1}{n-1}\lambda.
	\end{aligned}
	\end{equation}

In fact, \eqref{5.14} and \eqref{5.15} can be easily derived from  \eqref{4.1}.  On the other hand, 
	\begin{equation*}
	\begin{aligned}
	(XY)'=&X'Y+XY'\\
	=&\left(X'+X^2+\frac{1}{m}f'X\right)(X+Y) 
	-XY\left( X+Y+\frac{2}{m}f'\right) \\
	=&\left(X'+X^2+\frac{1}{m}f'X+XY\right)(X+Y) 
	-2XY\left( X+Y+\frac{1}{m}f'\right),
	\end{aligned}
\end{equation*}
yielding equation \eqref{5.16}. For equation \eqref{5.17}, we have 
\begin{equation*}
	\begin{aligned}
		&\left[(r_1-1)X+(r_2-1)Y\right]'=\left[(r_1-1)X'+(r_2-1)Y'\right]\\
		=&(n-3)\left(X'+X^2+\frac{1}{m}f'X+XY\right)(X+Y)\\
		&- [(r_1-1)X+(r_2-1)Y]\left( X+Y+\frac{1}{m}f'\right). 
	\end{aligned}
\end{equation*}
From \eqref{4.2} and \eqref{4.1}, we have	
\begin{equation*}
	\begin{aligned}
f''=&(n-1)\left( \xi'_a+\xi^2_a\right)+\frac{1}{m}f'^2+\lambda\\
   =&(n-1)\left(X'+X^2+\frac{1}{m}f'X \right)-\frac{1}{m}f'\sum^n_{i=2}\xi_a+\frac{1}{m}f'^2+\lambda\\
=&(n-1)\left(X'+X^2+\frac{1}{m}f'X+XY\right)-\frac{1}{m}f'[(r_1-1)X+(r_2-1)Y]\\
&-(n-1)XY-\frac{1}{m}f'(X+Y)+\frac{1}{m}f'^2+\lambda,	
	\end{aligned}
\end{equation*}
giving rise to equation \eqref{5.18}. Finally, it is readily seen from Equation \eqref{5.6} that \eqref{5.19} holds.

Now, by differentiating \eqref{5.10} and using the above identities, it follows that 	 
	\begin{equation}\label{5.20}
	\begin{aligned}
	-2&\left[ (r_1-1)\frac{k_1}{h^2_1}X + (r_2-1)\frac{k_2}{h^2_2}Y\right]\\
	=&  \frac{-n+3}{m}\left(X'+X^2+\frac{1}{m}f'X+XY\right)(X+Y)\\
	&+ \frac{2}{m}\left(X'+X^2+\frac{1}{m}f'X+XY\right)[(r_1-1)X+(r_2-1)Y]\\
	&-\frac{4(m+n-2)}{n-1}f'\left(X'+X^2+\frac{1}{m}f'X+XY\right)\\
	&-4(m+1)XY\left( X+Y+\frac{1}{m}f'\right)\\
        &-2[(r_1-1)X+(r_2-1)Y](X+Y)\left( X+Y+\frac{1}{m}f'\right)\\
\end{aligned}
	\end{equation}
\begin{equation*}
	\begin{aligned}
        &+2\frac{(m-1)}{m}[(r_1-1)X+(r_2-1)Y]XY\\
	&+\frac{(m+1)(2m+n-3)}{m}(X+Y)XY\\
	&+2\frac{(m-1)}{m^2(n-1)}f'[(r_1-1)X+(r_2-1)Y]^2\\
	&+\frac{(m+1)(2m+n-3)}{m^2(n-1)}f'(X+Y)[(r_1-1)X+(r_2-1)Y]\\ 
	&+2\frac{(m-1)}{m^2(n-1)}f'[(r_1-1)X+(r_2-1)Y](X+Y)\\
	&+\frac{(m+1)(2m+n-3)}{m^2(n-1)}f'(X+Y)^2\\
	& +\frac{(m+1)(2m+n-3)}{m(n-1)}f'(X+Y)^2\\
	& +2\frac{(m-1)}{m(n-1)}f'[(r_1-1)X+(r_2-1)Y](X+Y)\\
	&-2\frac{(m-1)}{m(n-1)}[(r_1-1)X+(r_2-1)Y]\lambda\\
	& -\frac{(m+1)(2m+n-3)}{m(n-1)}(X+Y)\lambda.
	\end{aligned}
	\end{equation*}

	Plugging \eqref{5.13} into \eqref{5.20}, gives us
	\begin{equation}\label{5.21}
	\begin{aligned}
	0=& \frac{2m-n+3}{m}\Big\{ (m+1)XY-\frac{m-1}{m(n-1)}f'[(r_1-1)X+(r_2-1)Y]\\
	&\quad\quad\quad\quad\quad\quad -\frac{m^2-1}{m(n-1)}f'(X+Y)+\frac{m-1}{n-1}\lambda \Bigg\}(X+Y)\\
	&+\frac{2}{m}\Bigg\{(m+1)XY-\frac{m-1}{m(n-1)}f'[(r_1-1)X+(r_2-1)Y]\\
  &\quad\quad\quad\quad-\frac{m^2-1}{m(n-1)}f'(X+Y)+\frac{m-1}{n-1}\lambda \Bigg\}[(r_1-1)X+(r_2-1)Y]\\
	&-\frac{4(m+n-2)}{n-1}f'\left(X'+X^2+\frac{1}{m}f'X+XY\right)\\
	&-4(m+1)XY\left( X+Y+\frac{1}{m}f'\right)\\
	&-2[(r_1-1)X+(r_2-1)Y](X+Y)\left( X+Y+\frac{1}{m}f'\right)\\
	&+2\frac{(m-1)}{m}[(r_1-1)X+(r_2-1)Y]XY\\
	&+\frac{(m+1)(2m+n-3)}{m}(X+Y)XY\\
	&+2\frac{(m-1)}{m^2(n-1)}f'[(r_1-1)X+(r_2-1)Y]^2\\	
\end{aligned}
	\end{equation}
\begin{equation*}
\begin{aligned}
       &+\frac{(m+1)(2m+n-3)}{m^2(n-1)}f'(X+Y)[(r_1-1)X+(r_2-1)Y]\\ 
	&+2\frac{(m-1)}{m^2(n-1)}f'[(r_1-1)X+(r_2-1)Y](X+Y)\\
	&+\frac{(m+1)(2m+n-3)}{m^2(n-1)}f'(X+Y)^2\\
	&+2\frac{(m-1)(m+1)}{m(n-1)}f'[(r_1-1)X+(r_2-1)Y](X+Y)\\
	&+\frac{(m-1)(m+1)(2m+n-3)}{m^2(n-1)}f'(X+Y)^2\\
	&-2\frac{(m-1)}{m(n-1)}[(r_1-1)X+(r_2-1)Y]\lambda\\
	& -\frac{(m+1)(2m+n-3)}{m(n-1)}(X+Y)\lambda\\
	&+2\left\{[(r_1-1)X+(r_2-1)Y]- \frac{m+1}{m}f' \right\}(X+Y)^2\\
	&-4\left[(r_1-1)X+(r_2-1)Y-\frac{m+1}{m}f'\right]XY+2\lambda(X+Y).
	\end{aligned}
	\end{equation*}
	
By simplifying equation \eqref{5.21}, we find that
	all terms cancel out,  except the third term
	\[
	-\frac{4(m+n-2)}{n-1}f'\left(X'+X^2+\frac{1}{m}f'X+XY\right).
	\]
	Hence, \eqref{5.7} holds (when $m+n-2\neq0$). 
Finally, plugging \eqref{5.7} into \eqref{5.19} yields \eqref{5.8}. This competes the proof of Lemma \ref{lemma5.1}. 
\end{proof}

\begin{lemma}\label{lemma5.2} Under the same assumption as in Lemma \ref{lemma5.1}, 
	we have 
	\begin{equation}\label{5.22}
	XY\left[(r_1-1)X+(r_2-1)Y-\frac{m+1}{m}f'\right]=0.
	\end{equation}	
\end{lemma}

\begin{proof}
	
	 First of all, by \eqref{5.7} and \eqref{5.8} in Lemma \ref{lemma5.1}, equations \eqref{5.16}-\eqref{5.18} become 
	\begin{equation}\label{5.23}
	(X-Y)'=-(X-Y)\left( X+Y+\frac{1}{m}f'\right),
	\end{equation}
	\begin{equation}\label{5.24}
	(XY)'=-2XY\left( X+Y+\frac{1}{m}f'\right),
	\end{equation}
	\begin{equation}\label{5.25}
	[(r_1-1)X+(r_2-1)Y]'=- [(r_1-1)X+(r_2-1)Y]\left( X+Y+\frac{1}{m}f'\right),
	\end{equation}
	and
	\begin{equation}\label{5.26}
	\begin{aligned}
	f''=&-\frac{1}{m}f'[(r_1-1)X+(r_2-1)Y]\\
	&-(n-1)XY-\frac{1}{m}f'(X+Y)+\frac{1}{m}f'^2+\lambda.
	\end{aligned}
	\end{equation}
	On the other hand, \eqref{5.8} can be rewritten as  
	\begin{equation}\label{5.27}
	f'[(r_1-1)X+(r_2-1)Y]+(m+1)f'(X+Y)
	=\frac{m(m+1)(n-1)}{(m-1)}XY+m\lambda.
	\end{equation}
	Hence, by substituting\eqref{5.27} into \eqref{5.26}, we get 
	\begin{equation}\label{5.28}
	f''=f'(X+Y)-\frac{2m(n-1)}{m-1}XY+\frac{1}{m}f'^2.
	\end{equation}
	Moreover, by differentiating \eqref{5.27},  then plugging \eqref{5.22}-\eqref{5.25} and \eqref{5.28} into it,
	and simplifying, we obtain \eqref{5.22}.
\end{proof}

Next, we divide our discussions into two cases. 

\subsection {{\bf Case I:} Either $X$ or $Y$ vanishes}

In this case, without loss of generality,  we assume that $X\neq0$ and $Y=0$. Then, as $Y=h_2'/h_2$, it follows that  $h_2(s)$ is constant.

\medskip
{\bf Claim:} $r_1\leq (n-3)$.

\medskip
We argue by contradiction. Suppose not. Then,  since $r_1\leq (n-3)$ is equivalent to $r_2\geq2$, we may assume that $r_2=1$. It then follows from \eqref{4.4}, $r_2=1$ and $Y=0$, that $ \lambda=\lambda_{\alpha}=0$.
However, \eqref{5.8} and $Y=0$ give us
\[
\lambda=\frac{m+r_1}{m}f'X=\frac{m+n-2}{m}f'X.
\]
Thus, when $m\neq 2-n$, either $f'$ or $X$ vanishes, which is a contradiction because, by assumption,   $X\neq 0$ and $(M^n, g)$ is not Einstein.

\begin{theorem} \label{two1}
	Let $(M^n, g, f)$, $n\geq 4$, be an $n$-dimensional quasi-Einstein manifold satisfying \eqref{qe} with harmonic Weyl curvature, where $m\notin{\{\pm1,\,(2-n),\, \infty\}}$. Suppose that, for each point $p\in M_A \cap \{ \nabla f \neq 0  \}$, there exists a neighborhood $U\subset M_A \cap \{ \nabla f \neq 0  \}$ of $p$ such that the Ricci eigenvalues $\lambda_2, \cdots, \lambda_n$  have exactly two distinct vlaues, $\lambda_a$ and $\lambda_\alpha$, of multiplicities $r_1=r$ and $r_2:=n-r-1$, respectively. If  $Y\!:=\xi_{\alpha}=0$, then $(U, g_{|U})$ is isometric to a domain in 
	\[ \left(\mathbb{R}^{1}\times N^{r}_1\times N_2^{n-r-1}, \ ds^2 + h^2_1(s) \tilde{g}_1+ \tilde{g_2}\right).
	\]
	Here,  $\left(N_2^{n-r-1},\tilde{g_2}\right) $ is an Einstein manifold with  Einstein constant $\lambda$; $(N^r_1, \tilde{g}_{1})$ is an Einstein manifold with Einstein constant $(r -1)k_1$ if $2 \leq r\leq (n-3)$, or $(N^1, \tilde{g}_{1})$ is given by $(L^1, \tilde{g}_{1}=dt^2)$ if $r=1$. 
	Moreover, 
\[\left(W^{r+1}, \bar{g}\right) =\left(\mathbb{R}^{1}\times N^{r}_1,\, ds^2 + h^2_1\tilde{g}_1\right) \] is a (D-flat) quasi-Einstein manifold satisfying \eqref{qe} which is also Einstein, with the Einstein constant $\rho\neq0$ satisfying $\frac{m+r}{r}\rho=\lambda$, and 
	of one of the following cases with $\Lambda:=\frac{1}{r}\rho$: 

\begin{enumerate}	
	\item[{\rm (i)}] 
	$\Lambda>0$, 
	$h_1=\sin({\sqrt{\Lambda}s})$,
	$f=-m\ln(\cos\sqrt{\Lambda}s)$ (up to an additive constant) and $k_1=\Lambda$. 
	In particular, for $1\leq r\leq 3$, $(M^n, g)$ is locally isometric to 
	a domain in $\mathbb{D}^{r+1} \times N_2^{n- r-1}$, 
	where $\mathbb{D}^{r+1}$ is the northern hemisphere 
	in the $(r+1)$-dimensional sphere $\mathbb{S}^{r+1} $ with the metric
	$ds^2+\sqrt{\Lambda}\sin^2\Big(\sqrt{\Lambda}s\Big)g_{\mathbb{S}^{r}}$, and
	$s\in(0,\frac{\pi}{2})$ is the distance function 
	on $\mathbb{B}^{r+1}$ from the north pole.
	
	\smallskip
	\item[{\rm (ii)}]   $\Lambda <0$, $h_1=\cosh({\sqrt{-\Lambda}s})$, $f=-m\ln|\sinh\sqrt{-\Lambda}s|$ and $k_1=\Lambda$.
	In particular, for $1\leq r\leq 3$, $(M^n, g)$ is locally isometric to 
	a domain in $\mathbb{D}^{ r+1} \times N_2^{n- r-1}$,
	where $\mathbb{B}^{ r+1}$ is the set $\{(s,t)|s<0\}$ 
	in the $(r+1)$-dimensional hyperbolic space $\mathbb{H}^{ r+1}$
	with the metric
	$ds^2+ \sqrt{-\Lambda}\cosh^2\left( \sqrt{-\Lambda}s\right) g_{\mathbb{H}^{ r}}$, 
	and $s\in(-\infty,0)$ 
	can be viewed as the signed distance function on $\mathbb{B}^{r+1}$ from the line $\{(s,t)|s = 0\}$.
	
	\smallskip	
	\item[{\rm (iii)}] 	
	$\Lambda <0$, $h_1=e^{\sqrt{-\Lambda}s}$, $f=-m\sqrt{-\Lambda}s$ and
	$(N^{r}, \tilde{g}_{1})$ is Ricci flat ($k_1=0$) for  $1\leq r\leq 3$.
	In particular, for $1\leq r\leq 3$, $(M^n, g)$ is locally isometric to  
	a domain in $\mathbb{R}^{ r+1} \times N_2^{n-r-1}$, 
	where $ds^2+e^{2\sqrt{-\Lambda}s}g_{\mathbb{R}^{ r}}$ is the metric on $\mathbb{R}^{ r+1}$ 	and $s\in(-\infty,+\infty)$ can be viewed as the signed distance function on $\mathbb{R}^{ r+1}$ from any base point.

	\smallskip
	\item[{\rm (iv)}] $\Lambda <0$, $h_1=\sinh({\sqrt{-\Lambda}s})$, $f=-m\ln(\cosh\sqrt{-\Lambda}s)$ and $k_1=-\Lambda$.
	In particular, for $1\leq r\leq 3$, $(M^n, g)$ is locally isometric to 
	a domain in $\mathbb{H}^{r+1} \times N_2^{n- r-1}$, 
	where $\mathbb H^{ r+1} $ is the $(r+1)$-dimensional hyperbolic space with the metric
	$ds^2+ \sqrt{-\Lambda}\sinh^2\left( \sqrt{-\Lambda}s\right) g_{\mathbb{S}^{ r}}$, 
	and $s\in(-\infty,0)$ 
	can be viewed as the signed distance function on $\mathbb{H}^{r+1}$ from any base point.
\end{enumerate}	
	
\end{theorem}

\begin{proof} 
	It follows from \eqref{5.8} and $Y=0$ that 
 \begin{equation}\label{f'X}
	\lambda=\frac{m+r}{m}f'X. \qquad (r=r_1)
 \end{equation}
	On the other hand, from \eqref{5.28}  and $Y=0$, we have 
	\[
	f''=f'X+\frac{1}{m}f'^2,	
		\]
	and then, by \eqref{3.13}, 
	\[
	\lambda_{1}=-f''+\frac{1}{m}f'^2+\lambda=-f'X+\lambda=\frac{r}{m}f'X.
	\]	
Moreover, it follows from the definitions of $X=\xi_a$ and $Y=\xi_{\alpha}$ (or \eqref{4.3} and \eqref{4.4})  that		
 \begin{equation*}
\lambda_a= -f'X+\lambda 
=\frac{r}{m}f'X 
\qquad{\rm and} \qquad \lambda_{\alpha}=\lambda.
 \end{equation*}
Therefore, the Ricci tensor is given by 
 \begin{equation}\label{R11=Raa}	
 R_{11} = R_{aa} =\frac{r}{m}f'X,
	 \end{equation}
\[ R_{\alpha\alpha}=\lambda, \quad R_{ij} =0\  (i \neq j).\]

In addition, we observe that 
$\left(W^{r+1}, \bar{g}\right) =\left(\mathbb{R}^{1}\times N^{r}_1,\, ds^2 + h^2_1\tilde{g}_1\right) $ is necessarily Einstein. 
In fact, for $r>2$, \eqref{R11=Raa} shows that $\frac{r}{m}f'X$ is a (nonzero) constant, denoted by $\rho$, hence  $W^{r+1}$ is Einstein with the Einstein constant $\rho$.  Also, 
\[
\lambda=\frac{m+r}{m}f'X=\frac{m+r}{r}\rho.
\]
For  $r=1$, \eqref{f'X} implies that $\lambda \neq 0$ since $m\neq-1$ and $f'X\neq0$. Thus 
	\[ R_{11} = R_{aa} =\frac{1}{m}f'X=\frac{1}{m+1}\lambda=\rho,\]
and $W^{2}$ is Einstein with the Einstein constant $\rho$. 

Next, from \eqref{4.2} and \eqref{4.3}, we have 
\[
f_{1,1}=f''=\frac{1}{m}f'^2+f'X, \quad f_{a,a}=f'X, \quad f_{a,b}=f_{1,a}=0.
\]
Therefore, $\left(W^{r+1}, \bar{g}, f\right)$ is also a D-flat quasi-Einstein manifold satisfying \eqref{qe}.
	
Furthermore, by \eqref{5.7}, 
\[
0=X'+X^2+\frac{1}{m}f'X=X'+X^2+\frac{1}{r}\rho,
\]
i.e.,
\[
X'+X^2+\Lambda=0,
\]
where $\Lambda=\frac{1}{r}\rho\neq0$. 

By solving the above ODE in $X$, we get either $X=\sqrt{-\Lambda}$ (if $\Lambda<0$) or 
\begin{align}\label{cases}
  X=\left\{
      \begin{array}{ll}
        -\sqrt{\Lambda}\tan(\sqrt{\Lambda}s+C), & \hbox{$\Lambda>0$,} \\
        \sqrt{-\Lambda}\tanh(\sqrt{-\Lambda}s+C), & \hbox{$\Lambda<0$,} \\
        \sqrt{-\Lambda}\coth(\sqrt{-\Lambda}s+C), & \hbox{$\Lambda<0$,}
           \end{array}
    \right.
\end{align}
where $C$ is an arbitrary constant. By choosing a suitable constant $C$ in \eqref{cases}, and using the relations $f'X=\frac{r}{m}\rho=m\Lambda$  and $X=\frac{h'_1}{h_1}$, we can obtain $X$, $f$ and $h_1$ as follows:

{\rm (i)~}  For $\Lambda>0$,
\[
X=\sqrt{\Lambda}\cot(\sqrt{\Lambda}s), \qquad	h_1=\sin({\sqrt{\Lambda}s}), 
\qquad f=-m\ln(\cos\sqrt{\Lambda}s).
\]

{\rm (ii)~}  For $\Lambda<0$,
\[
X=\sqrt{-\Lambda}\tanh(\sqrt{-\Lambda}s), \quad	h_1=\cosh({\sqrt{-\Lambda}s}),  
\quad f=-m\ln|\sin\sqrt{-\Lambda}s|.
\]

{\rm (iii)~}  For $\Lambda<0$,
\[
X=\sqrt{-\Lambda}, \quad	h_1=e^{\sqrt{-\Lambda}s},  \quad f=-m\sqrt{-\Lambda}s.
\]

{\rm (iv)~}  For $\Lambda<0$,
\[
X=\sqrt{-\Lambda}\coth(\sqrt{-\Lambda}s), \quad	h_1=\sinh({\sqrt{-\Lambda}s}), 
\quad f=-m\ln(\cosh\sqrt{-\Lambda}s).
\]

By making use of the above information, it follows from \eqref{4.3} that		
\begin{equation*}
	\begin{aligned}
		-f'X+\lambda=&-\left( X'+X^2+\frac{1}{m}f'X\right)+\frac{1}{m}f'X+X^2+(r-1)\frac{k_1}{h^2_1}-rX^2\\
		=&\Lambda+(r-1)\frac{k_1}{h^2_1}-(r-1)X^2
	\end{aligned}	
\end{equation*}	
and 	
\[
-f'X+\lambda=-f'X+\frac{m+r}{m}f'X=\frac{r}{m}f'X=r\Lambda,
\]
which imply that	
\[
k_1=(X^2+\Lambda)h^2_1.
\]
Therefore, $k_1=\Lambda$ for (i) and (ii), while $k_1=-\Lambda$ for (iv). 
For (iii), $k_1=0$, and then $(N^r_1, \tilde{g}_{1})$ is Ricci flat.	 This finishes the proof of Theorem~\ref{two1}
\end{proof}

\subsection {{\bf Case II:}   Neither $X$ nor $Y$ vanishes} 
For this subcase, we shall first show the following useful lemma, which will lead to the desired local structure result. 

\begin{lemma}\label{lemma5.1.1}
	Let $(M^n, g, f)$, $n\geq 4$, be an $n$-dimensional quasi-Einstein manifold satisfying \eqref{qe} with harmonic Weyl curvature and $m\notin{0, \, \{\pm1,\,(2-n),\, \infty\}}$.
	Suppose that, in a neighborhood $U$ of $p$ in $M_A \cap \{ \nabla f \neq 0  \}$, 
	 the Ricci eigenvalues 
	$\lambda_2, \cdots, \lambda_n$  have exactly two distinct values, 
	denoted by $\lambda_a$ and $\lambda_\alpha$, of multiplicities $r_1$ and $r_2:=n-r_1-1$, respectively.
	If neither $X$ nor $Y$ vanishes,
	then
	\begin{equation}\label{5.1.1}
	(r_1-1)X+(r_2-1)Y-\frac{m+1}{m}f'=0,
	\end{equation}
	\begin{equation}\label{5.1.2}
	(r_1-1)\frac{k_1}{h^2_1}= (r_2-1)\frac{k_2}{h^2_2}=\lambda=0,
	\end{equation}
	\begin{equation}\label{5.1.3}
	f'\left( X+Y+\frac{1}{m}f'\right)=\frac{m(n-1)}{m-1}XY,
	\end{equation}	
	\begin{equation}\label{5.1.4}
	\begin{aligned}
	\frac{(r_1-1)(m+r_1)}{m+1}X^2 & +\frac{(r_2-1)(m+r_2)}{m+1}Y^2\\
	&+2\left[	\frac{(r_1-1)(r_2-1)}{m+1}-\frac{m+n-2}{m-1} \right]XY=0
	\end{aligned}
	\end{equation}
	and	
	\begin{equation}\label{5.1.5}
	R=-\frac{m-1}{m}f'^2.
	\end{equation}	
\end{lemma}

\begin{proof}
	First of all, \eqref{5.1.1} holds from Lemma \ref{lemma5.2}. 
Putting \eqref{5.7} and \eqref{5.1.1} into \eqref{5.11} and \eqref{5.12} yields
	   \begin{equation}\label{5.37}
        (r_1-1)\frac{k_1}{h^2_1}=\lambda \qquad {\rm and} 
       \qquad  (r_2-1)\frac{k_2}{h^2_2}=\lambda,
   \end{equation} 
respectively.
Then, by differentiating \eqref{5.1.1} and using \eqref{5.25} and \eqref{5.28}, it follows that 	 	
		\begin{equation*}
				\begin{aligned}
-& [(r_1-1)X+(r_2-1)Y]\left( X+Y+\frac{1}{m}f'\right)\\
& =\frac{m+1}{m}\left\lbrace f'( X+Y)-\frac{2m(n-1)}{m-1}XY+\frac{1}{m}f'^2\right\rbrace.
	\end{aligned}
	\end{equation*}
Simplifying the above equation, we get \eqref{5.1.3}.	

Next, by plugging \eqref{5.1.1} into \eqref{5.27}, we have
  \begin{equation}\label{5.38}
 \frac{m+1}{m}f'^2+(m+1)f'(X+Y) = \frac{m(m+1)(n-1)}{(m-1)}XY+m\lambda
 \end{equation}
On the other hand, by using \eqref{5.1.3}, the left hand side of \eqref{5.37} is 
	\begin{equation}\label{5.39}
	\begin{aligned}
        \frac{m+1}{m}f'^2+(m+1)f'(X+Y) = & (m+1)f'\left( X+Y+\frac{1}{m}f'\right) \\
       = & \frac{m(m+1)(n-1)}{m-1}XY.
		\end{aligned}
\end{equation}
Thus, it follows immediately from \eqref{5.38} and \eqref{5.39} that $\lambda=0$. This and \eqref {5.37} together imply \eqref{5.1.2}.

By using \eqref{5.1.1} and \eqref{5.1.3}, we can also get\eqref{5.1.4}.	
		Finally,  by substituting \eqref{5.1.3} into \eqref{5.28}, we have 
	\begin{equation}\label{5.29}
	f''=-f'(X+Y)-\frac{1}{m}f'^2,
	\end{equation}
	Combining  \eqref{5.29} and \eqref{4.2} yields
	\[
	\lambda_{1}=f'(X+Y)+\frac{2}{m}f'^2.
	\]
	Moreover, \eqref{4.3} and \eqref{4.4} imply that $\lambda_{a}=-f'X$ and $\lambda_{\alpha}=-f'Y$.
	Hence, 
	\begin{equation*}
	\begin{aligned}
	R=&\lambda_{1}+r_1\lambda_{a}+r_2\lambda_{\alpha}\\
	=&f'(X+Y)+\frac{2}{m}f'^2-r_1f'X-r_2f'Y\\
	=&-f'\left[ (r_1-1)X+(r_2-1)Y\right]+\frac{2}{m}f'^2\\
	=&-\frac{m-1}{m}f'^2,
	\end{aligned}
	\end{equation*}	
	where we have used \eqref{5.1.1} in the last equality.  This completes the proof of the lemma. 
\end{proof}

\begin{theorem}\label{two2}
	Let $(M^n, g, f)$, $n\geq 4$, be an $n$-dimensional quasi-Einstein manifold satisfying \eqref{qe} with harmonic Weyl curvature, where $m\notin{\{\pm1,\,(2-n),\, \infty\}}$. Suppose that, for each point $p\in M_A \cap \{ \nabla f \neq 0  \}$, there exists a neighborhood $U\subset M_A \cap \{ \nabla f \neq 0  \}$ of $p$ such that the Ricci eigenvalues $\lambda_2, \cdots, \lambda_n$  have exactly two distinct vlaues, $\lambda_a$ and $\lambda_\alpha$, of multiplicities $r_1$ and $r_2:=n-r_1-1$, respectively. If neither $X$ nor $Y$ vanishes, then we have $\lambda=0$ and $(U, g_{|U})$ is isometric to one of the following:
	\begin{enumerate}
	
	\item [(i)]
	a domain in 
	\[
	\left( L\times N^{r_1}_1\times N^{r_2}_2, \ 
	ds^2 + (b_1s+c_1)^{\frac{2}{b_1}} \tilde{g}_{1}+(b_2s+c_2)^{\frac{2}{b_2}} \tilde{g}_{2}\right), 
	\]
	with $f=b_3\log (b_1s+c_1)$, modulo a constant, for some constant $c_1$.
	Here, $L$ and $N^1_1$ (when $ r_1=1$) are open intervals; $(N^{r_1}_1, \tilde{g}_{1})$ (for $2 \leq r_1\leq n-2 $) and $(N_2, \tilde{g}_{2})$ are both Ricci-flat;
	and $b_i~(1\leq i\leq 3)$ are constants, which are determined by the coefficients in \eqref{5.1.4} and depend only on $n$, $m$ and $r_1$.

	\smallskip
	\item [(ii)] a domain in
	\[
	\left( \mathbb R\times N^{r_1}_1\times N^{r_2}_2, \ 
	ds^2 + e^{2c_3s} \tilde{g}_{1}+e^{2b_4c_3s} \tilde{g}_{2}\right)
	\]
for $2 \leq r_1\leq n-3 $ and some constant $c_3\neq0$, with $f=\frac{mc_3[(r_1-1)+(r_2-1)b_4]}{m+1}s$ modulo a constant.	
Here, $(N_1, \tilde{g}_{1})$ and $(N_2, \tilde{g}_{2})$ are both Ricci-flat,
and the constant $b_4\notin\{0,1\}$ is determined by the coefficients in \eqref{5.1.4}, depending only on $n$, $m$ and $r_1$.
	
	\end{enumerate}
\end{theorem}

\begin{proof}
	Firstly, since either $r_1\geq 2$ or $r_2\geq 2$, without loss of generality, we may assume $r_2\geq2$. 
	Then, equation \eqref{5.1.4} can be considered as a quadratic function in $Y$. 
		By direct computations, we find its discriminant is given by 
	\[
	\frac{4(m+n-2)}{(m-1)^2}\left[m(n-1)-(m-1)r_1r_2 \right]. 
	\]
 Hence, 
 	\begin{equation*}\label{5.1.6}
 (m+n-2)\left[m(n-1)-(m-1)r_1r_2 \right]>0.
 \end{equation*}

	\smallskip

	If  $r_1=1$ then, from \eqref{5.1.4}, we get that
	\begin{equation*}
	(m+n-2)Y\left[ 2X-\frac{(m-1)(n-3)}{m+1}Y\right] =0. 
	\end{equation*}
In particular, since $m\neq 2-n$ and $Y\neq 0$, it follows that 
	\begin{equation}\label{5.1.7}
	 2X-\frac{(m-1)(n-3)}{m+1}Y =0,
	\end{equation}
	and then, by \eqref{5.1.1}, we have
	\begin{equation}\label{5.1.8}
	(n-3)Y-\frac{m+1}{m}f'=0.
	\end{equation}	
	It is easy to see that $(n-3)\neq \frac{2(m+1)}{m-1}$ from \eqref{5.1.7}, because $X\neq Y$.
	From \eqref {5.1.7} and \eqref{5.1.8}, it follows that 
        \begin{equation}\label{5.1.9}
                 f'=\frac{m(n-3)} {m+1}Y =\frac{2m}{m-1}X. 
        \end{equation}	
Combining \eqref{5.1.9} with \eqref{5.7},  we have 
	\[
	X'+\frac{(m+1)(n-1)}{(m-1)(n-3)}X^2=0. 
	\]
	Solving the above equation, we obtain 
	\[ X=\frac{1}{bs+c_1}, \qquad {\rm with}\qquad  b=\frac{(m+1)(n-1)}{(m-1)(n-3)},\]
for an arbitrary constant $c_1$. 
	
Moreover, by using the equations $f'=\frac{2m}{m-1}X$,  $\frac{h'_1}{h_1}=X$ and $\frac{h'_2}{h_2}=Y$,  we can easily solve for functions $f$, $h_1$ and $h_2$.  
 Therefor, over the dense open subset,  $(M^n,g)$ is locally isometric to a domain in 
 	\[
 	\left( \mathbb{R}^+\times L\times N^{n-2},
 	ds^2 + (bs+c_1)^{\frac{2}{b}} dt^2+  \left(\frac{n-1}{2}s+c_2\right)^{\frac{4}{n-1}} \tilde{g}\right) 
 	\]
 	for some constant $c_2$, where $\left( N^{n-2} ,\tilde{g}\right)$ is Ricci flat and 
 	$$f=\frac{2m(n-3)}{(m+1)(n-1)} \log (bs+c_1) $$
 	 modulo a constant.

	\smallskip
For $r_1\geq2$, we can apply the same argument as in the case $r=1$:  
by the real-analyticity of quasi-Einstein manifolds, equation \eqref{5.1.4} shows that a certain linear combination of $X$ and $Y$ is equal to zero. This allows us to handle the case in a similar manner as above. 

 Indeed, by combining \eqref{5.1.1} with \eqref{5.7} and using the fact that $X$ and $Y$ are linearly related, we have
 \[
 X'+b_1X^2=0. 
 \] 
for some constant $b_1$, which is determined by the coefficients in \eqref{5.1.4} and depends only on $n$, $m$ and $r_1$.		
If $b_1\neq0$, then we have 
\[
X=\frac{1}{b_1s+c_1} \qquad {\rm and } \qquad Y=\frac{1}{b_2s+c_2}
\]
for some constants $c_1$, $c_2$ and $b_2\neq 0$, 
where $b_1s+c_1>0$, $b_2s+c_2>0$, and $b_2$ is also determined by the coefficients in \eqref{5.1.4}, depending only on $n$, $m$ and $r_1$.

Therefore, $(U, g_{|U})$ is  isometric to a domain in
\[
\left( L\times N^{r_1}_1\times N^{r_2}_2,
g= ds^2 + (b_1s+c_1)^{\frac{2}{b_1}} \tilde{g}_{1}+(b_2s+c_2)^{\frac{2}{b_2}} \tilde{g}_{2}\right), 
\]
with $f=b_3\log (b_1s+c_1)$ modulo a constant, for some constant $b_3$, which is determined by the coefficients in \eqref{5.1.4}, depends only on $n$, $m$ and $r_1$. 
Combining the case $r_1=1$, we have type (i) in Theorem \ref{two2}.
 \smallskip
 
If $b_1=0$, then
$X=c_3$ for some constant $c_3\neq0$. Thus,
$Y=b_4c_3$, where constant $b_4\notin\{0,1\}$, which is also determined by the coefficients in \eqref{5.1.4}, depending only on $n$, $m$ and $r_1$.
Therefore, in this case, $(U, g_{|U})$ is isometric to a domain in
\[
\left(\mathbb R\times N^{r_1}_1\times N^{r_2}_2,
g= ds^2 + e^{2c_3s} \tilde{g}_{1}+e^{2b_4c_3s} \tilde{g}_{2}\right), 
\]
with $f=\frac{m[(r_1-1)c3+(r_2-1)c4]}{m+1}s$ modulo a constant. This is type (ii) in Theorem \ref{two2} and we have completed the proof of Theorem \ref{two2}.
\end{proof}

\section {The local structure theorem and the global classifcations}

In this section, we summarize our main results obtained in Section 3 and Section 4 to get the local structure theorem, Theorem \ref{local}, and then discuss the global classifcations, including Theorem \ref{complete'} as stated in the introduction.

\subsection{The local structure theorem} By combining Theorem \ref{thm4.1}, Theorem\ref{same}, Theorem \ref{two1} and Theorem \ref{two2}, we have the following local structure theorem. 

\begin{theorem} \label{local}
	Let $(M^n, g, f)$, $n\geq 4$, be an $n$-dimensional quasi-Einstein manifold satisfying \eqref{qe} with harmonic Weyl curvature and $m>1$ ($m\neq \infty$). Then,  each point $p\in M_A \cap \{ \nabla f \neq 0  \}$ admits a neighborhood $V_p$ such that $(V_p, g_{|V_p})$ is isometric to one of the following spaces: 

\smallskip
\begin{enumerate}
	\item[{\rm (a)}] either an Einstein manifold (with a constant potential function $f$) or a warped product with 
	\[
	g=ds^2+h^2(s)\bar{g} 
	\]
	for some positive function $h(s)$, where the Riemannian metric $\bar g$ is Einstein. In particular, $g$ is $D$-flat.

	\medskip
	\item[{\rm (b)}]	
	$\left( W^{\bar r},\bar{g}\right) \times\left(  N^{n-\bar r},\tilde{g}\right)$ for $2\leq\bar r\leq n-2$.
	Here, $\left(N^{n-\bar r},\tilde{g}\right) $ is an Einstein manifold with Einstein constant $\lambda$; $\left(W^{\bar r}, \bar{g}\right)=\left(L ,\, ds^2 \right) \times\, _{h_1}\left(N^{\bar r-1}_1, \tilde{g}_{1} \right) $ is a D-falt quasi-Einstein manifold satisfying \eqref{qe}, which is also 
	$\rho$-Einstein ($\rho\neq0$), where $L$ is an open interval, $\left(N^{\bar r-1}_1, \tilde{g}_{1} \right)$ is Einstein with Einstein constant $(\bar r -2)k_1$ for $3 \leq \bar r\leq (n-2)$ and $\lambda=\frac{m+\bar r-1}{\bar r-1}\rho$. 
	Moreover, with $\Lambda:=\frac{1}{\bar r-1}\rho$, one of the following holds
	\begin{enumerate}	
		\item[{\rm (i)}] 
	$\Lambda>0$, 
	$h_1=\sin({\sqrt{\Lambda}s})$,
	 $f=-m\ln(\cos\sqrt{\Lambda}s)$ (up to an additive constant) and $k_1=\Lambda$. 
	In particular, for $2\leq \bar r\leq 4$, $(V_p, g_{|V_p})$ is isometric to 
	a domain in $\mathbb{S}^{\bar r}_{+} \times N^{n-\bar r}$, 
	where $\mathbb{S}^{\bar r}_{+}$ is the northern hemisphere 
	in the $\bar r$-dimensional sphere $\mathbb{S}^{\bar r} $ with the metric
	$ds^2+\sqrt{\Lambda}\sin^2\Big(\sqrt{\Lambda}s\Big)g_{\mathbb{S}^{\bar r-1}}$, and
	$s\in(0,\frac{\pi}{2})$ is the distance function 
	on $\mathbb{S}^r_{+}$ from the north pole.
	
	\smallskip
	\item[{\rm (ii)}]  $\Lambda <0$, $h_1=\cosh({\sqrt{-\Lambda}s})$, $f=-m\ln|\sinh\sqrt{-\Lambda}s|$ and $k_1=\Lambda$.
	In particular, for $2\leq \bar r\leq 4$, $(V_p, g_{|V_p})$ is isometric to 
	a domain in $\mathbb{B}^{\bar r} \times N^{n-\bar r}$,
	where $\mathbb{B}^{\bar r}$ is the set $\{(s,t)|s<0\}$ 
	in the $\bar r$-dimensional hyperbolic space $\mathbb{H}^{\bar r}$
	with the metric
	$ds^2+\sqrt{-\Lambda}\cosh^2\left( \sqrt{-\Lambda}s\right) g_{\mathbb{H}^{\bar r-1}}$, 
	and $s\in(-\infty,0)$ 
	can be viewed as the signed distance function on $\mathbb{B}^{\bar r}$ from the line $\{(s,t)|s = 0\}$.
	
	\smallskip	
	\item[{\rm (iii)}] 	
	$\Lambda <0$, $h_1=e^{\sqrt{-\Lambda}s}$, $f=-m\sqrt{-\Lambda}s$ and
 $(N^{\bar r-1}, \tilde{g}_{1})$ is Ricci flat ($k_1=0$) for $3 \leq \bar r\leq (n-2)$.
	In particular, for $2\leq \bar r\leq 4$, $(V_p, g_{|V_p})$ is isometric to 
	a domain in $\mathbb{R}^{\bar r} \times N^{n-\bar r}$, 
		where $ds^2+e^{2\sqrt{-\Lambda}s}g_{\mathbb{R}^{\bar r-1}}$ is the metric on $\mathbb{R}^{\bar r}$ 
	and $s\in(-\infty,+\infty)$ can be viewed as the signed distance function on $\mathbb{R}^{\bar r}$ from any base point.

	\smallskip
	\item[{\rm (iv)}]  $\Lambda <0$, $h_1=\sinh({\sqrt{-\Lambda}s})$, $f=-m\ln(\cosh\sqrt{-\Lambda}s)$ and $k_1=-\Lambda$.
	In particular, for $2\leq \bar r\leq 4$, $(V_p, g_{|V_p})$ is isometric to 
	a domain in $\mathbb{H}^{\bar r} \times N^{n-\bar r}$, 
	where $\mathbb H^{\bar r} $ is the $\bar r$-dimensional hyperbolic space with the metric
	$ds^2+\sqrt{-\Lambda}\sinh^2\left( \sqrt{-\Lambda}s\right) g_{\mathbb{S}^{\bar r-1}}$, 
	and $s\in(-\infty,0)$ 
	can be viewed as the signed distance function on $\mathbb{H}^{\bar r}$ from any base point.
	\end{enumerate}
	
\medskip
	\item[{\rm (c)}] $\lambda=0$ and either 
	
	\begin{enumerate}
	\item [(i)]
	a domain in 
	\[
\quad \quad\quad\quad\quad\quad\quad	\left( L\times N^{r}_1\times N^{n-r-1}_2,
	g= ds^2 + (b_1s+c_1)^{\frac{2}{b_1}} \tilde{g}_{1}+(b_2s+c_2)^{\frac{2}{b_2}} \tilde{g}_{2}\right), 
	\]
	with $f=b_3\log (b_1s+c_1)$ modulo a constant, for some constant $c_1$.
	Here, $L$ and $N^1_1$ (when $r=1$) are open intervals; $(N^{r}_1, \tilde{g}_{1})$ (for $2 \leq r\leq n-2 $) and $(N_2, \tilde{g}_{2})$ are Ricci-flat; 
	and $b_i~(1\leq i\leq 3)$ are three constants, which are determined by the coefficients in the following equation \eqref{5.1.4'}, depending only on $n$, $m$ and $r$.	
	\begin{equation}\label{5.1.4'}
	\begin{aligned}
\quad\quad	\frac{(r-1)(m+r)}{m+1}x^2 & +2\left[	\frac{(r-1)(n-r-2)}{m+1}-\frac{m+n-2}{m-1} \right]xy\\
	&+\frac{(n-r-2)(m+n-r-1)}{m+1}y^2=0
	\end{aligned}
	\end{equation}	
	for $xy\neq 0$ and $x\neq y$; or

	\smallskip
	\item [(ii)] a domain in 
	\[
\quad\quad\quad\quad	\left(\mathbb R\times N^{r}_1\times N^{n-r-1}_2,
	g= ds^2 + e^{2c_3s} \tilde{g}_{1}+e^{2b_4c_3s} \tilde{g}_{2}\right)
	\]
	for $2 \leq r\leq n-3 $ and some constant $c_3\neq0$, with 
	$$\quad\quad\quad\quad\quad\quad f=\frac{mc_3[(r-1)+(n-r-2)b_4]}{m+1}s$$
	 modulo a constant.	
	Here, both $(N_1, \tilde{g}_{1})$ and $(N_2, \tilde{g}_{2})$ are Ricci-flat, 
	and constant $b_4\notin\{0,1\}$, which is determined by the coefficients in \eqref{5.1.4'}, depends only on $n$, $m$ and $r$.
		\end{enumerate}	
	\end{enumerate}		
\end{theorem}

\begin{remark}
	For $n=3$, He, Petersen and Wylie \cite[Example 3.5]{HPW} showed that there are local solutions to the quasi-Einstein equation \eqref{qe} which are not locally conformally flat and are doubly warped product metrics of the form
	\[
	g = \mathrm{d}s^{2}+\phi(s)^{2}\mathrm{d}\theta_{1}^{2}+\psi(s)^{2}\mathrm{d}\theta_{2}^{2}.
	\]
	Meanwhile, it is worth noting that quasi-Einstein manifolds of types (i)-(iv) for (b) and types (i)-(ii) for (c) in Theorem \ref{local}  provide new examples of neither locally conformally flat nor $D$-flat ones. 
\end{remark}

\subsection{The global classifications} First, we show that case (c) in Theorem \ref{local} does not occur when the quasi-Einstein manifold $(M^n, g, f)$ is complete (with or without boundary). 

\begin{proposition}\label{two2.1}
	Let $(M^n, g, f)$, $n\geq 4$, be an $n$-dimensional quasi-Einstein manifold satisfying \eqref{qe} with harmonic Weyl curvature and $m>1$ ($m\neq\infty$). If $(M^n, g)$ is locally isometric to a space of either type \!(c)(i) or type \!(c)(ii) in Theorem \ref{local}, then it cannot be complete (with or without boundary). 
\end{proposition}

\begin{proof}
	First of all, by setting $w=e^{-\frac{f}{m}}$, we claim that the constant 
	\[
	\mu=w\Delta w+(m-1)|\nabla w|^2+\lambda w^2
	\]
(as in Kim-Kim \cite[Proposition 5]{KK}) must vanish.
In fact, 
\begin{equation*}
\begin{aligned}
\mu=&w\Delta w+(m-1)|\nabla w|^2+\lambda w^2\\
=&\frac{1}{m}e^{-\frac{2f}{m}}\left[-\Delta f+|\nabla f|^2 +m\lambda \right]\\
=&\frac{1}{m}e^{-\frac{2f}{m}}\left( R+\frac{m-1}{m}f'^2 \right),
\end{aligned}
\end{equation*}		
where we have used $R+ \Delta f-\frac 1 m |\nabla f|^2 =n\lambda$, and the fact that $\lambda=0$ for type \!(c), in the last equality. 

Hence, from \eqref{5.1.5}, it immediately follows that $\mu=0$.
However, by Case \cite[Theorem 1.1]{case} and He-Petersen-Wylie \cite[Corollary 4.2]{HPW}, 
any complete $(\lambda,n+m)$-Einstein manifold with $m>1$, $\lambda\geq 0$, and $\mu\leq 0$ is necessarily trivial.
Thus, the quasi-Einstein manifold $(M^n, g, f)$ in the proposition cannot be complete.
\end{proof}

Meanwhile,  Case-Shu-Wei \cite[Proposition 4.2]{CSW} provided the following classification of complete $(\lambda, n+m)$-Einstein manifolds which are also Einstein with some Einstein constant $\rho \neq \lambda $  (called $\rho$-Einstein in \cite{HPW}). 

\begin{proposition}{\rm (\cite[Proposition 4.2]{CSW})}\label{propkappaEinstein'} 
A complete, non-trivial, $(\lambda, m+n)$-Einstein manifold $(W^n, \bar g, w=e^{-f/m})$, with $n \geq 2$ and $m>0$ ($m\neq\infty$), is $\rho$-Einstein for some $\rho\neq \lambda$ if and only if $W^n$ is diffeomorphic to ${\mathbb R}^n$ and given by either  
\begin{enumerate}

\smallskip
\item[{\rm (A)}] 	
$\lambda <0$ and
$(W^n, \bar g)$ is isometric to
$$\left({\mathbb R}\times F^{n-1}, \ ds^{2}+ e^{2\sqrt{-\bar{k}}s}g_{F}\right),$$ with $w(s) =e^{\sqrt{-\bar{k}}s}$, or 

\smallskip
\item[{\rm (B)}] 	
$\lambda<0$ and $(W^n, \bar g)$ is isometric to the hyperbolic space
$$\quad\quad \left(\mathbb{H}^n, \ ds^{2}+ \sqrt{-\bar{k}} \sinh^2(\sqrt{-\bar{k}} s) g_{\mathbb{S}^{n-1}}\right),$$ 
with $w(s) =  \cosh(\sqrt{-\bar{k}}s)$.
\end{enumerate}
 Here, $\bar{\kappa} = \frac{\lambda - \rho}{m} \neq 0$, $(F, g_F)$ is Ricci-flat, and $\mathbb{S}^{n-1}$ is the round unit sphere.
\end{proposition}

Subsequently, He-Petersen-Wylie \cite[Proposition 3.1]{HPW} extended the work of Case-Shu-Wei \cite{CSW} to the case of $(W^n, \bar g, w=e^{-f/m})$ being complete, with or without boundary, as follows:

\begin{proposition}{\rm (\cite[Proposition 3.1]{HPW})}\label{propkappaEinstein} 
Let $(W^n, \bar g, w=e^{-f/m})$, $n \geq 2$, be  a complete (with or without boundary) non-trivial $(\lambda, m+n)$-Einstein manifold which is also $\rho$-Einstein, for some $\rho\neq \lambda$,  and $m>0$ ($m\neq\infty$). Then, it is one of the following cases:
\begin{enumerate}
\item[{\rm (i)}]  
$\lambda >0$, 
$(W^n, \bar g)$ is isometric to a quotient of the upper hemisphere $$\left({\mathbb S}^{n}_{+},  \ ds^{2}+ \sqrt{\bar{k}} \sin ^{2}(\sqrt{\bar{k}} s) g_{\mathbb{S}^{n-1}}\right)
 \quad {\rm and}\quad  w(s) = \cos(\sqrt{\bar{k}}s);$$

\smallskip	
\item[{\rm (ii)}]  
$\lambda =0$, $(W^n, \bar g)$ is isometric to 
$$\left([0, \infty) \times F^{n-1}, \ ds^{2}+ g_{F}\right)
 \quad {\rm and}  \quad w(s)=s;$$

\smallskip
\item[{\rm (iii)}] 	
$\lambda<0$ and $(W^n, \bar g)$ is isometric to 
$$\left(  [0, \infty) \times N^{n-1}, \ ds^{2}+ \sqrt{-\bar{k}} \cosh ^{2}(\sqrt{-\bar{k}} s) g_{N}\right)
 \quad {\rm and} \quad w(s) =  \sinh(\sqrt{-\bar{k}}s); $$

\smallskip
\item[{\rm (iv)}] 	
$\lambda <0$ and
$(W^n, \bar g)$ is isometric to
$$\left({\mathbb R} \times F^{n-1}, \ ds^{2}+ e^{2\sqrt{-\bar{k}}s}g_{F}\right)
\quad {\rm and} 
\quad  w(s) =e^{\sqrt{-\bar{k}}s};$$

\smallskip
\item[{\rm (v)}] 	
$\lambda<0$ and $(W^n, \bar g)$ is isometric to 
$$\quad\quad \quad \left(\mathbb{H}^n, \ ds^{2}+ \sqrt{-\bar{k}} \sinh^2(\sqrt{-\bar{k}} s) g_{\mathbb{S}^{n-1}}\right)
 \quad {\rm and } \quad w(s) =  \cosh(\sqrt{-\bar{k}}s).$$
\end{enumerate}
 Here, $\bar{\kappa} = \frac{\lambda - \rho}{m} \neq 0$, $\mathbb{S}^{n-1}$ is the round unit sphere, $(F, g_F)$ is Ricci flat, and $(N, g_N)$ is Einstein  with negative Ricci curvature.
\end{proposition}

Now, we are ready to sketch a proof of Theorem \ref{complete'}: 

\begin{enumerate}

\medskip
\item[$\bullet$] {\bf Step 1:} By Proposition \ref{two2.1}, we know that locally, over the open dense subset $M_A \cap \{ \nabla f \neq 0  \}\subset M$, $(M^n, g, f)$ is of types (a) and (b) in Theorem \ref{local}. 

\smallskip
\item[$\bullet$] {\bf Step 2:} Suppose $(M^n, g, f)$ is non-trivial and of type (a), then it is $D$-flat. Therefore, by \cite[Theorem 1.2]{HPW} (or by using a similar argument as in \cite{FLGR, CY, Kim2}),  it is globally a warped product as in  part (1) of Theorem \ref{complete'}. 

\smallskip
\item[$\bullet$] {\bf Step 3:} Suppose $(M^n, g, f)$ is non-trivial and of type (b) in Theorem \ref{local}. We consider its universal cover $(\tilde{M}^n, {\tilde g}, {\tilde f})$. Then, by the completeness and simply-connectedness of $(\tilde{M}^n, {\tilde g}, {\tilde f})$, and the analyticity of ${\tilde g}$, one can show that  $(\tilde{M}^n, {\tilde g})$ is globally a Riemannian product of Einstein manifolds  $(M_1^k, g_1)$ and $(M_2^{n-k}, g_2)$, $2\leq k \leq n-2$, with Einstein constants $\rho$ and $\lambda\neq 0$, respectively ($\bar{\kappa} = \frac{\lambda - \rho}{m}=\Lambda$). 
Furthermore, ${\tilde f}$ is a function on $M_1^k$ only and  $(M_1^k, g_1, {\tilde f})$ is a quai-Einstein manifold satisfying \eqref{qe}. Therefore, by Proposition \ref{propkappaEinstein'}, $(M_1^k, g_1, {\tilde f})$ is of type (A) or type (B), thus  proving parts (2) and (3) of Theorem \ref{complete'}. 
\end{enumerate}

\begin{remark} 
If we assume $(M^n, g, f)$ in Theorem \ref{complete'}  to be complete, possibly with boundary, then, in view of  Proposition \ref{propkappaEinstein}, we would have the following corresponding result: 
\end{remark}

\begin{theorem}\label{complete}
	Let $(M^n, g, f)$, $n\geq 5$, be an $n$-dimensional complete, with or without boundary, non-trivial quasi-Einstein manifold satisfying \eqref{qe} with harmonic Weyl curvature and $m>1$ and $m\neq \infty$. Then, it is one of the following types:
	\begin{enumerate}
	\item[{\rm (a)}] $(M^n, g, f)$ is a quotient of some warped product quasi-Einstein manifold of the form
	\[
	\left(\mathbb{R} ,\, ds^2 \right) \times\, _h\left(N^{n-1}, \bar{g} \right), 
	\]
	where $\left(N^{n-1}, \bar{g} \right) $ is an Einstein manifold. [This is the $D$-flat case.]  
	
\smallskip
\item[{\rm (b)}] $\lambda\neq 0$ and $(M^n,g)$ is isometric to a quotient of the Riemannian product 
\[\hskip 1in \left( W^{\bar r},\bar{g}\right) \times\left(  N^{n-{\bar r}},\tilde{g}\right), \qquad 2\leq {\bar r}\leq n-2.\]
Here, $\left(N^{n-{\bar r}},\tilde{g}\right) $ is an Einstein manifold with Einstein constant $\lambda$;  $\left(W^{\bar r}, \bar{g}\right)$ is a (D-falt) quasi-Einstein manifold satisfying \eqref{qe} which is also an Einstein manifold with Einstein constant $\rho=\frac{{\bar r}-1}{m+{\bar r}-1}\lambda$, given by one of the following cases with $\Lambda =\frac{1}{m+{\bar r}-1}\lambda$:
\begin{enumerate}	
	\item[{\rm (i)}]  
	$\Lambda >0$ and $\left(W^{\bar r}, \bar{g}\right) $ is isometric to 
	 $$\left(\mathbb{S}^{\bar r}_{+}, \ ds^{2}+ \sqrt{\Lambda} \sin ^{2}(\sqrt{\Lambda} s) g_{\mathbb{S}^{{\bar r}-1}}\right),$$ where $\mathbb{S}^{\bar r}_{+}$ is an upper hemisphere, and (up to an additive constant) $f= -m\log(\cos\sqrt{\Lambda}s)$.
	
	\smallskip
	\item[{\rm (ii)}] 	
	$\Lambda<0$ and $\left(W^{\bar r}, \bar{g}\right) $ is isometric to
	$$\quad \quad \quad \left([0, \infty) \times \bar{N}^{{\bar r}-1},   \ ds^{2}+ \sqrt{-\Lambda} \cosh ^{2}(\sqrt{-\Lambda} s) g_{\bar{N}}\right), $$ 
	where $(\bar{N}^{{\bar r}-1}, g_{\bar{N}})$ is an Einstein metric with negative scalar curvature, and $f = -m\log (\sinh\sqrt{-\Lambda}s)$.
	
	\smallskip
	\item[{\rm (iii)}] 	
	$\Lambda <0$ and
	$\left(W^{\bar r}, \bar{g}\right) $ is isometric to 
	$$\left(\mathbb{R}\times F^{{\bar r}-1}, \ ds^{2}+ e^{2\sqrt{-\Lambda}s}g_{F}\right),$$ where $(F^{{\bar r}-1}, g_{F})$ is Ricci flat, and $f =-m{\sqrt{-\Lambda}s}$.
	
	\smallskip
	\item[{\rm (iv)}] 	
	$\Lambda<0$ and $\left(W^{\bar r}, \bar{g}\right) $ is isometric to 
	$$\left(\mathbb{H}^{{\bar r}}, \ ds^{2}+ \sqrt{-\Lambda} \sinh^2(\sqrt{-\Lambda} s) g_{\mathbb{S}^{\bar r-1}}\right),$$ 
	where $\mathbb{H}^{\bar r}$ is a hyperbolic space, and $f= -m\log(\cosh \sqrt{-\Lambda}s)$.
\end{enumerate}		
	\end{enumerate}	
\end{theorem}

\bigskip
\noindent {\bf Acknowledgments.} We would like to thank Dr. Junming Xie for helpful discussions.  Part of this work was carried out while the second author was visiting Lehigh University from August, 2019 to August, 2020. She would like to thank the Department of Mathematics at Lehigh University for its hospitality and for providing an excellent research  environment. She would also like to thank Professor Yu Zheng and Professor Zhen Guo for their constant encouragements and support.


\medskip

\begin{thebibliography}{99}
	
	\bibitem{Anderson} 
	M.~T.~Anderson, 
	\emph{Scalar curvature, metric degenerations and the static vacuum Einstein equations 
        on $3$-manifolds. I}, 
	Geom. Funct. Anal. {\bf 9} (1999), no.~5, 855--967.
	
	\bibitem{AndersonKhuri} 
	M.~Anderson, M.~Khuri, 
	\emph{On the Bartnik extension problem for the static vacuum Einstein equations},
         Classical Quantum Gravity. {\bf 30} (2013), no. 12, 125005, 33 pp.
	
	\bibitem{BGKW} 
	E.~Bahuaud, S.~Gunasekaran, H.~Kunduri, E.~Woolgar, 
	\emph{Static near-horizon geometries and rigidity of quasi-Einstein manifolds}, 
	Lett. Math. Phys. {\bf 112} (2022), no.~6, Paper~No.~116, 16~pp.
	
	\bibitem{BE} 
	D.~Bakry, M.~\'Emery, 
	\emph{Diffusions hypercontractives}, 
	in: S\'eminaire de probabilit\'es~XIX, 1983/84, 177--206, 
	Lecture Notes in Math. {\bf 1123}, Springer, Berlin, 1985.
	
	\bibitem{Besse} 
	A.~L.~Besse, 
	\emph{Einstein Manifolds}, 
	Ergebnisse der Mathematik und ihrer Grenzgebiete~(3), Vol.~10, 
	Springer--Verlag, Berlin, 1987.
		
	\bibitem{BHS2025} 
	V.~Borges, M.~A.~R.~M.~Hor\'acio, J.~P.~dos~Santos,
	\emph{The Ricci tensor of a gradient Ricci soliton with harmonic Weyl tensor}, 
	arXiv:2510.11939.
	
	\bibitem{Brinkmann} 
	H.~W.~Brinkmann, 
	\emph{Einstein spaces which are mapped conformally on each other}, 
	Math. Ann. {\bf 94} (1925), 119--145.
		
	\bibitem{Cao2010} 
	H.-D.~Cao, 
	\emph{Recent progress on Ricci solitons}, 
	in: \emph{Recent Advances in Geometric Analysis}, 1--38, 
	Adv. Lect. Math.~(ALM) {\bf 11}, Int. Press, Somerville, MA, 2010.

	\bibitem{CC12} 
	H.-D.~Cao, Q.~Chen, 
	\emph{On locally conformally flat gradient steady solitons}, 
	Trans. Amer. Math. Soc. {\bf 364} (2012), no.~5, 2377--2391. 
	
	\bibitem{CC13} 
	H.-D.~Cao, Q.~Chen, 
	\emph{On Bach-flat gradient shrinking Ricci solitons}, 
	Duke Math. J. {\bf 162} (2013), no.~6, 1149--1169. 

            \bibitem{CY} 
            H.-D.~Cao, J.~Yu, \emph {On complete gradient steady Ricci solitons with 
             vanishing D-tensor}, Proc. Amer. Math. Soc. {\bf 149} (2021), no. 4, 1733--1742.
	
	\bibitem{case} 
	J.~Case, 
	\emph{The nonexistence of quasi-Einstein metrics}, 
	Pacific J. Math. {\bf 248} (2010), no.~2, 277--284. 

	\bibitem{CSW} 
	J.~Case, Y.-J.~Shu, G.~Wei, 
	\emph{Rigidity of quasi-Einstein metrics}, 
	Differential Geom. Appl. {\bf 29} (2011), 93--100.
	
	\bibitem{Ca} 
	G.~Catino, 
	\emph{Generalized quasi-Einstein manifolds with harmonic Weyl tensor}, 
	Math. Z. {\bf 271} (2012), no.~3--4, 751--756.
	
	\bibitem{CMMR} 
	G.~Catino, C.~Mantegazza, L.~Mazzieri, M.~Rimoldi, 
	\emph{Locally conformally flat quasi-Einstein manifolds}, 
	J. Reine Angew. Math. {\bf 675} (2013), 181--189.
	
	\bibitem{CH} 
	Q.~Chen, C.~He, 
	\emph{On Bach-flat warped product Einstein manifolds}, 
	Pacific J. Math. {\bf 265} (2013), no.~2, 313--326.
	
	\bibitem{Corvino} 
	J.~Corvino, 
	\emph{Scalar curvature deformation and a gluing construction for the 
         Einstein constraint equations}, 
	Comm. Math. Phys. {\bf 214} (2000), no.~1, 137--189.
	
	\bibitem{De} 
	A.~Derdzi\'nski, 
	\emph{Classification of certain compact Riemannian manifolds with harmonic 
       curvature and non-parallel Ricci tensor}, 
	Math. Z. {\bf 172} (1980), 273--280.

	\bibitem{FLGR} M.~Fern\'andez-L\'opez and E.~García-R\'io, 
         \emph{A note on locally conformally flat gradient Ricci solitons}, 
        Geom. Dedicata {\bf 168} (2014), 1--7. 

	\bibitem{HPW} 
	C.~He, P.~Petersen, W.~Wylie, 
	\emph{On the classification of warped product Einstein metrics}, 
	Comm. Anal. Geom. {\bf 20} (2012), no.~2, 271--311.
	
	\bibitem{HW} 
	G.~Huang, Y.~Wei, 
	\emph{The classification of $(m,\rho)$-quasi-Einstein manifolds}, 
	Ann. Global Anal. Geom. {\bf 44} (2013), 269--282.
	
	\bibitem{Israel} 
	W.~Israel, 
	\emph{Event horizons in static vacuum space-times}, 
	Phys. Rev. {\bf 164} (1967), no.~5, 1776--1779.
	
	\bibitem{JW} 
	J.~Jauregui, W.~Wylie, 
	\emph{Conformal diffeomorphisms of gradient Ricci solitons and generalized 
       quasi-Einstein manifolds}, 
	J. Geom. Anal. {\bf 25} (2015), no.~1, 668--708.
	
	\bibitem{KK} 
	D.-S.~Kim, Y.-H.~Kim, 
	\emph{Compact Einstein warped product spaces with nonpositive scalar curvature}, 
	Proc. Amer. Math. Soc. {\bf 131} (2003), no.~8, 2573--2576.
	
	\bibitem{Kim} 
	J.~Kim, 
	\emph{On a classification of $4$-dimensional gradient Ricci solitons 
        with harmonic Weyl curvature}, 
	J. Geom. Anal. {\bf 27} (2017), no.~2, 986--1012.
	
	\bibitem{Kim2} 
	J.~Kim, 
	\emph{Classification of gradient Ricci solitons with harmonic Weyl curvature}, 
	J. Geom. Anal. {\bf 35} (2025), no.~5, Paper~No.~139, 33~pp.
	
	\bibitem{Li} 
	F.~Li, 
	\emph{Rigidity of complete gradient steady solitons with harmonic Weyl tensor}, 
	Pacific J. Math. {\bf 335} (2025), no.~2, 323--353.
	
	\bibitem{LPP} 
	H.~L\"u, D.~N.~Page, C.~N.~Pope, 
	\emph{New inhomogeneous Einstein metrics on sphere bundles over 
        Einstein-K\"ahler manifolds}, 
	Phys. Lett. B {\bf 593} (2004), 218--226.
	
	\bibitem{MR} 
	P.~Mastrolia, M.~Rimoldi, 
	\emph{Some triviality results for quasi-Einstein manifolds and Einstein warped products}, 
	Geom. Dedicata {\bf 169} (2014), 225--237.
	
	\bibitem{Qian} 
	Z.~Qian, 
	\emph{Estimates for weighted volumes and applications}, 
	Quart. J. Math. Oxford Ser. (2) {\bf 48} (1997), no.~190, 235--242.
	
	\bibitem{Shin} 
	J.~Shin, 
	\emph{On the classification of four-dimensional $(m,\lambda)$-quasi-Einstein 
        manifolds with harmonic Weyl curvature}, 
	Ann. Global Anal. Geom. {\bf 51} (2017), 379--399.
	
	\bibitem{WW} 
	G.~Wei, W.~Wylie, 
	\emph{Comparison geometry for the Bakry--\'Emery Ricci tensor}, 
	J. Differential Geom. {\bf 83} (2009), no.~2, 377--405.

         \bibitem{Wylie 2023} W.~Wylie,
          \emph {Rigidity of compact static near-horizon geometries with negative 
           cosmological constant},  
         Lett. Math. Phys. {\bf 113} (2023), no. 2, Paper No. 29, 5 pp.

\end{thebibliography}
\end{document}